\documentclass[reqno,10pt]{amsart}
\makeatletter
 \usepackage{enumitem}
\newcommand{\Rmnum}[1]{\expandafter\@slowromancap\romannumeral#1@}

\newtheorem{lemma}{Lemma}[section]

\newtheorem{theorem}{Theorem}[section]

\usepackage{amsmath}
\allowdisplaybreaks[4]
\usepackage{amsmath}
\usepackage{amssymb}
\usepackage{amsfonts}
\usepackage{mathrsfs}
\usepackage{color}
\usepackage{cite}
\numberwithin{equation}{section}
\DeclareMathOperator{\dive}{div}

\title[Compressible Navier-Stokes-Poisson System]
{Stability of planar shock wave for the 3-dimensional compressible Navier-Stokes-Poisson  equations}%
\author[X. Wu]{}
\email{xcwu22@csu.edu.cn}

\begin{document}

\maketitle \centerline{\scshape  Xiaochun Wu}
\medskip
{\footnotesize
   \centerline{$\ \!$School of Mathematics and Statistics, HNP-LAMA, Central South University }
\centerline{Changsha 410083, China}
}

\medskip
\bigskip
\begin{abstract}
This paper is concerned with the stability of planar viscous shock wave for the 3-dimensional compressible Navier-Stokes-Poisson (NSP) system in the domain $\Omega:=\mathbb{R}\times \mathbb{T}^2$ with $\mathbb{T}^2=(\mathbb{R}/\mathbb{Z})^2$. The stability problem of viscous shock under small 1-dimensional perturbations was solved in Duan-Liu-Zhang \cite{Duan-Liu-Zhang}. In this paper, we prove the viscous shock is still stable under small 3-d perturbations. 
Firstly, we decompose the perturbation into the zero mode and non-zero mode. Then we can show that both the perturbation and zero-mode time-asymptotically tend to zero by the anti-derivative technique and crucial estimates on the zero-mode. Moreover, we can further prove that  the non-zero mode tends to zero with exponential decay rate. The key point is to estimate the non-zero mode of nonlinear terms involving electronic potential, see Lemma \ref{L:6.1} below.
\end{abstract}

\section{Introduction}
\subsection{Background}

The 3-d two fluid Navier-Stokes-Poisson system is a typical model describing the transport of charged particles affected by the self-consistent electrostatic potential force in collisional dusty plasma, which is written as
\begin{align}\label{1}
	\left\{\begin{array}{l}
		\partial_t \rho_i+\dive\left(\rho_i {\bf u}_i\right)=0, \\
	 m_i\rho_i\left(\partial_t {\bf u}_i+{\bf u}_i\cdot \nabla {\bf u}_i\right)+T_i\nabla \rho_i=\mu_i \Delta {\bf u}_i+\rho_i \nabla E, \\
		\partial_t \rho_e+\dive\left(\rho_e {\bf u}_e\right)=0, \\
	 m_e\rho_e\left(\partial_t {\bf u}_e+{\bf u}_e\cdot \nabla {\bf u}_e\right)+T_e\nabla \rho_e=\mu_e \Delta {\bf u}_e-\rho_e \nabla E,  \\
		\lambda^2\Delta E=\rho_i-\rho_e.
	\end{array}\right.
\end{align}
Here, the variables $\rho_\alpha=\rho_\alpha(x, t): \Omega\times \mathbb{R}^+  \rightarrow \mathbb{R}^+$ and ${\bf u}_\alpha={\bf u}
_\alpha(x, t): \Omega\times \mathbb{R}^+  \rightarrow \mathbb{R}^3$ denote the density and velocity, respectively, of ions $(\alpha=i)$ and electrons $(\alpha=e)$, while $E=E(x, t)$ refers to the self-consistent potential. The constants $m_\alpha, T_\alpha$ and $\mu_\alpha$ are used to represent the mass, absolute temperature, and viscosity coefficient of the fluid of ions $(\alpha=i)$ and electrons $(\alpha=e)$, as well as $\lambda>0$ stands for Debye length. 

 Since the electrons in plasma approach equilibrium faster than the heavier ions, cf. \cite{Chen, Krall},
we consider the case of single ions flow under the Boltzmann relation. This model exhibits many interesting phenomena, such as traveling waves \cite{Kundu-Ghosh-Chatterjee-Das}, which have also attracted  considerable attentions.
As shown in \cite{ Guo-Pausader}, the Boltzmann relation $\rho_e=e^{-\frac{E}{T_e}}$ can be formally derived from the two-fluid model by assuming a zero velocity for the electrons. And then  
the 3-d normalized NSP system for ions reads (dropping the subscript $i$)
\begin{align}\label{E:1.1}
	 \left\{\begin{array}{l}
	 \partial_t \rho+\operatorname{div} \bf m=0, \\
	 \partial_t {\bf m}+\operatorname{div}\left(\frac{\bf m \otimes \bf m}{\rho}\right)+T\nabla \rho=\mu \Delta {\bf u}-\rho \nabla E, \quad x\in \Omega, \,t>0, \\
	 -\lambda^2 \Delta E=\rho-e^E,
		\end{array}\right.
	\end{align}
where the unknown variable ${\bf m}=\rho{\bf u}$ is the momentum. The initial values are given by 
\begin{align}\label{E:1.1b}
	\begin{cases}
		\rho(x,0)=\rho_0(x)\rightarrow \rho_{ \pm}>0 \quad \mbox{as} \quad x_1\rightarrow \pm \infty,\\
		{\bf m}(x,0)={\bf m}_0(x)\rightarrow (m_{1\pm}, 0, 0)^T \quad \mbox{as} \quad x_1\rightarrow \pm \infty,
	\end{cases}
\end{align} 
and the far-field data of $E$ is given by 
\begin{equation}\label{E:1.12b}
	\lim\limits_{x_1\rightarrow \pm \infty}E(x, t)=E_{\pm}
\end{equation}
under the quasi-neutral condition $E_{\pm}=\ln \rho_{ \pm}$ at $x_1\rightarrow \pm \infty$. In this paper, we are concerned with the Cauchy problem of 3-d NSP system \eqref{E:1.1} in 
 $x=(x_1, x_2, x_3)^T\in \Omega:=\mathbb{R}\times \mathbb{T}^2$ with $\mathbb{T}^2=(\mathbb{R}/\mathbb{Z})^2$. Without loss of generality, we assume that $\mu=\lambda=1$ throughout the paper.


 For the  multi-dimensional compressible NSP system \eqref{E:1.1}, 
Hao-Li \cite{Hao-Li} proved the global well-posedness of the smooth solution for the Cauchy problem of NSP in the Besov type space. 
Then, Li et al. \cite{Li-Matsumura-Zhang} not only provided the well-posedness in Sobolev space, but also described its asymptotic behavior. 
Indeed, they showed the smooth solution converges to its constant equilibrium state and obtained the optimal decay rate. Later, Zhang et al. \cite{Zhang-Li-Zhu} extended the result to the case of non-isentropic NSP system. For the works related to weak solution, we refer to 
\cite{Chae, D-F-P-S, Zhang-Tan, Kobayashi-Suzuki} and references therein.
For the two-fluid NSP system \eqref{1}, Li et al. \cite{Li-Yang-Zou} and Hsiao et al. \cite{Hsiao-Li-Yang-Zou} obtained the optimal decay rate for the isentropic and non-isentropic cases respectively.
More works related to the one-fluid/two-fluid NSP system, see \cite{Hong-Shi-Wang, Donatelli-Marcati, Wang-Wu, Tan-Yang-Zhao-Zou, Cai-Tan, Tan-Wang-Wang,  Jiang-Lai-Yin-Zhu, Cui-Gao-Yin-Zhang, Jang-Tice, Liu-Luo-Zhong}.

Note that the asymptotic profile mentioned above of multi-dimensional NSP system most is the stationary state, next we introduce some nonlinear stability of wave patterns for 1-dimensional. Duan-Yang \cite{Duan-Yang} firstly proved the stability of rarefaction wave and boundary layer for outflow problem of the two-fluid NSP system. 
Then, Li-Zhu \cite{Li-Zhu} investigated the outflow problem of NSP system for ions and showed the asymptotic stability, toward a nonlinear wave which is the superposition of a stationary solution and a rarefaction wave. 
 However, to the best of our knowledge, so far there are few mathematical results of the nonlinear stability of wave patterns for multi-dimensional NSP system. This is exactly the motivation of our paper.

\subsection{The problem and main Result}

As it well-known, the technical obstacle in the shock theory for multiple dimensions is the negative sign of $(\tilde u_1^s)^{\prime}<0$, which can not been overcome due to the loss of the anti-derivatives method. Recently, Yuan \cite{Yuan} introduced a decomposition idea: for any function $f\in L^{\infty}(\Omega)$ that is periodic in $x^{\prime}=(x_2, x_3)\in \mathbb T^2$, it can be decomposed into the 1-dimensional zero mode $\bar f$ and the multi-dimensional non-zero mode $f^{\neq}$ 
and still made a good use of the anti-derivative technique for multi-dimensional case in $\Omega=\mathbb R \times \mathbb T^2$, where 
\begin{align}\label{E:1.18}
\bar f(x_1):=\int_{{\mathbb T}^2}f(x_1, x^{\prime})dx^{\prime} \quad \mbox{and}\quad f^{\neq}(x):=f(x)-\bar f(x_1).
\end{align}
  Motivated by this, we will use the above decomposition to consider the stability of planar shock wave to 3-d NSP system.
  
For 1-d NSP system, \eqref{E:1.1}-\eqref{E:1.12b} can be replaced by 
\begin{align}\label{E:1.1c}
	\left\{\begin{array}{l}
		\partial_t\tilde \rho+\tilde m_{1x_1}=0,\\
		\tilde m_{1t}+\left(\frac{\tilde m_1^2}{\tilde \rho}\right)_{x_1}+T\tilde \rho_{x_1}=\tilde u_{1x_1x_1}-\tilde \rho \tilde E_{x_1},\\
		-\tilde E_{x_1x_1}=\tilde \rho-e^{\tilde E}
	\end{array}\right.
\end{align}
with 
\begin{align}\label{E:1.2c}
	\begin{cases}
		\tilde \rho(x_1,0)=\tilde \rho_0(x_1)\rightarrow \rho_{ \pm}>0 \quad \mbox{as} \quad x_1\rightarrow \pm \infty,\\
		\tilde m_1(x_1,0)=\tilde m_{10}(x_1)\rightarrow m_{1\pm} \quad \mbox{as} \quad x_1\rightarrow \pm \infty
	\end{cases}
\end{align}
and 
\begin{equation}\label{E:1.3cc}
	\lim\limits_{x_1\rightarrow \pm \infty}\tilde E(x_1, t)=E_{\pm}.
\end{equation}
As shown in \cite{Duan-Liu}, the two characteristics of the corresponding quasineutral Euler system of \eqref{E:1.1c}-\eqref{E:1.3cc} are 
\begin{align*}
	\lambda_{\pm}(\rho, u_1)=u_1\pm \sqrt{T+1} \quad \mbox{with} \quad u_1=\frac{ m_1}{ \rho}.
\end{align*}
 When the small-amplitude $|\rho_{-}-\rho_{+}|\ll 1$,  Duan et al. \cite{Duan-Liu-Zhang} showed that if the constants $[\rho_{ \pm}, u_{1\pm}]$ satisfy the Rankine-Hugoniot conditions
 \begin{align}\label{E:1.9}
 	\left\{\begin{array}{l}
 		-s\left(\rho_{+}-\rho_{-}\right)+\rho_+u_{1+} -\rho_-u_{1-}=0, \\
 		-s\left(\rho_+u_{1+} -\rho_-u_{1-}\right)+ \rho_+u_{1+}^2- \rho_-u_{1-}^2+(T+1)\left(\rho_{+}-\rho_{-}\right)=0,
 	\end{array}\right.
 \end{align}
 and the Lax shock condition (corresponding the second family, namely, 2-shock)
 \begin{align}\label{E:1.10}
 	\lambda_+(\rho_{+}, u_{1+})<s<	\lambda_+(\rho_{-}, u_{1-})\quad \mbox{and}\quad s>\lambda_-(\rho_{-}, u_{1-}),
 \end{align}
 then \eqref{E:1.1c}-\eqref{E:1.3cc} admits a  viscous 2-shock wave solution $(\tilde \rho, \tilde m_1, \tilde E)(x_1, t):=(\tilde \rho^s, \tilde m_1^s, \tilde E^s)(\xi)$ with $\xi=x_1-st$ satisfying 
\begin{equation*}
	\lim\limits_{x_1\rightarrow \pm \infty} (\tilde \rho, \tilde m_1, \tilde E)(x_1, t)=(\rho_{ \pm}, m_{1\pm}, E_{\pm}).
\end{equation*} 
Then it follows from \eqref{E:1.9}-\eqref{E:1.10} (see \cite{Smoller}) that 
\[\rho_->\rho_+ \quad \mbox{and}\quad u_{1-}>u_{1+}.\]
Define $\tilde {\bf m}(x,t)=(\tilde m_1(x_1,t), 0, 0)^T$, then such $(\tilde \rho,\tilde {\bf m},\tilde E)(x,t)$ are called the planar shock wave to the Cauchy problem of the 3-d system \eqref{E:1.1}-\eqref{E:1.12b}.


Define 
\begin{align}\label{E:1.7}
	(z_0, {\bf r}_0)(x)=(z_0, (r_{10}, r_{20}, r_{30})^T)(x):=(\rho-\tilde \rho, {\bf m}-\tilde {\bf m})(x, 0)
\end{align}
and 
\begin{align}\label{E:1.8}
	(Z_0, R_0)(x_1)=\int_{-\infty}^{x_1}\int_{\mathbb{T}^2}(z_0, r_{10})(y_1, x^{\prime})dx^{\prime}dy_1, \quad \quad x_1\in \mathbb{R}.
	\end{align}
 Then, we state the main theorem.
\begin{theorem}\label{theorem1.1}
	Assume that \eqref{E:1.9}-\eqref{E:1.10} hold. Then there exist $\delta>0$ and $\nu_1>0$ such that if $|\rho_{-}-\rho_{+}|\leq \delta$ and \\
	1. the perturbation $(z_0, {\bf r}_0)$ is periodic in the the transverse variables $x^{\prime}$ on $\mathbb{T}^2$; moreover, the pairs \eqref{E:1.7} and \eqref{E:1.8} satisfy that 
	\begin{equation}\label{zxx}
		\nu_0:=\|z_0, {\bf r}_0\|_{H^4(\Omega)}+\|Z_0, R_0\|_{L^2(\mathbb{R})}\leq \nu_1;
	\end{equation}
2. the following zero-mass type condition holds:
\begin{equation}\label{zc}
\int_{\Omega} z_0 d x=0 \quad \text { and } \quad \int_{\Omega} {\bf r}_0 d x={\bf 0}.
\end{equation}
Then the Cauchy problem \eqref{E:1.1}-\eqref{E:1.12b} admits a unique classical solution $(\rho, {\bf m}, E)(x,t)$ globally in time, which is periodic in $x^{\prime} \in \mathbb{T}^2$ and satisfies that
\begin{align}\label{E:1.14}
	(\rho, {\bf m})-(\tilde \rho, \tilde {\bf m}) \in C(0, +\infty; H^4(\Omega)), \quad E_t-\tilde E_t\in C(0, +\infty; H^5(\Omega))\quad \mbox{and}\quad E-\tilde E \in C(0, +\infty; H^6(\Omega)).
\end{align}
Moreover, it holds that 
\begin{align}\label{E:1.15}
	\|(\rho, {\bf m})-(\tilde \rho, \tilde {\bf m})\|_{W^{2, \infty}(\mathbb{R}^3)}+\|E_t-\tilde E_t\|_{W^{3, \infty}(\mathbb{R}^3)}+\|E-\tilde E\|_{W^{4, \infty}(\mathbb{R}^3)}\rightarrow 0 \quad \mbox{as} \quad t\rightarrow +\infty;
\end{align}
and the non-zero mode of the solution, 
\begin{align*}
	(\rho^{\neq}, {\bf m}^{\neq}, E^{\neq}, E_t^{\neq}):=(\rho, {\bf m}, E, E_t)(x,t)-\int_{\mathbb{T}^2} (\rho, {\bf m}, E, E_t)(x_1, x^{\prime}, t)d x^{\prime}, \quad x\in \Omega, \, t>0
\end{align*}
satisfies that 
\begin{align}\label{E:1.17}
	\|(\rho^{\neq}, {\bf m}^{\neq})(\cdot, t)\|_{W^{1, \infty}(\mathbb{R}^3)}+\|E_t^{\neq}(\cdot, t)\|_{W^{2, \infty}(\mathbb{R}^3)}+\|E^{\neq}(\cdot, t)\|_{W^{3, \infty}(\mathbb{R}^3)}\leq C\nu_0 e^{-ct}, \quad  \forall \,t>0
\end{align}
with some constants $C>0$ and $c>0$.
\end{theorem}


We now sketch the main strategy. Denote the perturbation variables by $(z, {\bf r}, H)=(\rho, {\bf m}, E)-(\tilde \rho, \tilde {\bf m}, \tilde E)$ and ${\bf w}={\bf u}-\tilde {\bf u}$. For 3-d perturbation system $(z, {\bf w}, H)$, we can only derive the following estimate due to the failure of the anti-derivative technique:
	\begin{align}\label{E:5.30n}
		&\|z, {\bf w},  H\|_{H^1(\Omega)}^2+\int_0^t\|\nabla z, \nabla {\bf w}, \nabla^2 {\bf w}\|_{L^2(\Omega)}^2d\tau
		\notag\\
		\lesssim &\mathbb{N}^2(0)+(\delta+\nu)\int_0^t\|z, w_1\|_{L^2(\Omega)}^2d\tau+(\delta+\nu)\int_0^t\|H\|_{H^1(\Omega)}^2d\tau
		\end{align}
under the {\it a priori} assumption \eqref{ass}. Here, the last two terms on the right hand side of \eqref{E:5.30n} are difficult to estimate.
	
To deal with the second term on the right hand side of \eqref{E:5.30n}, we need decompose the perturbation into the zero mode and non-zero mode. For the zero mode $(\bar z, \bar r_1, \bar H)$, using the zero-mass type condition \eqref{zc} we can define the following anti-derivates variables
\begin{align*}
(Z, R)(x_1, t)
=\int_{-\infty}^{x_1} (\bar z, \bar r_1)\left(y_1, t\right) d y_1.
	\end{align*}
Then, the perturbation system for zero mode $(Z, R, \bar H)$ is applied to get
		\begin{align}\label{sa}
		&\|R\|_{L^2(\mathbb{R})}^2+\|Z,\bar H\|_{H^1(\mathbb{R})}^2+\int_0^t \left(\|\sqrt{-\partial_1\tilde u_1}R,\partial_1R,\partial_1Z\|_{L^2(\mathbb{R})}^2+\|\bar H\|_{H^2(\mathbb{R})}^2\right)d\tau
		\notag\\
		\lesssim &\mathbb{N}^2(0)+\nu \|\nabla H\|_{L^2(\Omega)}^2
	 +(\delta+\nu) \int_0^t\left(\|z, r_1, H_t \|_{L^2(\Omega)}^2+\|H\|_{H^1(\Omega)}^2\right) d \tau.
	\end{align}
Thanks to the property of $\bar f$ shown in Lemma \ref{l1} and the relation between $w_1$ and $r_1$ shown in Lemma \ref{L:5.1n}, we get
\begin{align}\label{poq}
\|z, w_1, r_1\|_{L^2(\Omega)}^2\lesssim \|\partial_1 Z, \partial_1 R\|_{L^2(\mathbb{R})}^2+\|\nabla z, \nabla {\bf w}\|_{L^2(\Omega)}^2.
\end{align}
Thus, adding \eqref{E:5.30n} and \eqref{sa} up and using \eqref{poq} yields that
	\begin{align}\label{E:5.30nbm}
		&\|R\|_{L^2(\mathbb{R})}^2+\|Z,\bar H\|_{H^1(\mathbb{R})}^2+\|z, {\bf w},  H\|_{H^1(\Omega)}^2+\int_0^t \left(\|\partial_1R,\partial_1Z\|_{L^2(\mathbb{R})}^2+\|\bar H\|_{H^2(\mathbb{R})}^2\right)d\tau
		\notag\\
		&+\int_0^t\|\nabla z, \nabla {\bf w}, \nabla^2 {\bf w}\|_{L^2(\Omega)}^2d\tau
		\lesssim \mathbb{N}^2(0)+(\delta+\nu)\int_0^t\left(\|H_t\|_{L^2(\Omega)}^2+\|H\|_{H^1(\Omega)}^2\right)d\tau.
		\end{align}

It suffices to handle the term involving the perturbation electrostatic potential $H$ on the right hand side of \eqref{E:5.30nbm}. Based on the key observation from the structures of perturbation system that $\dive (\rho{\bf w})$ can be transformed into the derivatives of potential $H$:
\begin{equation}
\dive (\rho {\bf w})=-\partial_t z -\dive (z\tilde {\bf u}) \quad \mbox{and}\quad z=- \Delta H+N_1,\nonumber
\end{equation}
we choose some suitable weighted function to obtain the estimates of potential as follows:
\begin{align}\label{sa1}
		&\|H\|_{H^2(\Omega)}^2+\|H_t \|_{H^1(\Omega)}^2+\int_0^t\|H, H_t\|_{H^2(\Omega)}^2d\tau
	\lesssim \mathbb{N}^2(0)+\|\bar H\|_{L^2(\mathbb{R})}^2+\delta\|z\|_{L^2(\Omega)}^2
	\notag\\
	&+\int_0^t\|\bar H, \bar H_t\|_{L^2(\mathbb{R})}^2d\tau
	 +\left(\delta+\nu\right)\int_0^t\left(\|z\|_{H^1(\Omega)}^2+\|w_1, \nabla {\bf w}\|_{L^2(\Omega)}^2\right)d\tau.
\end{align} 
In addition, we can verify that 
	\begin{align}\label{E:4.37q}
	\|\bar H_t\|_{H^1(\mathbb{R})}^2\lesssim \|\partial_1 R\|_{L^2(\mathbb{R})}^2+\delta \|\bar H\|_{L^2(\mathbb{R})}^2+\nu\|\nabla H_t\|_{L^2(\Omega)}^2.
	\end{align}
Thus, taking the procedure as $c_0\eqref{E:5.30nbm}+\eqref{sa1}$ with some large constant $c_0>0$, we use  \eqref{poq} and \eqref{E:4.37q} to get 
\[\left\|Z, R, \bar H\right\|_{L^2(\mathbb{R})}+\|z, {\bf w}, H_t\|_{H^1(\Omega)}+\|H\|_{H^2(\Omega)}\lesssim C_0.\]
Moreover, we can further derive \eqref{E:1.15} after establishing the higher-order estimates of perturbation variables $(z,{\bf w}, H)$.	

Next, we will show the exponential decay rate of non-zero mode. It is worth to pointing out that the non-zero mode  satisfies
\begin{align}\label{qop}
	\| z^{\neq}, {\bf w}^{\neq}, H^{\neq}, H_t^{\neq}\|_{L^2(\Omega)}^2\lesssim\|\nabla z^{\neq}, \nabla {\bf w}^{\neq}, \nabla H^{\neq}, \nabla H_t^{\neq}\|_{L^2(\Omega)}^2,
\end{align}
which is the main reason why the anti-derivates method is not necessary and the desired decay rate can be derived. 
However, for the perturbation system of the non-zero mode, it is not trivial to show that the nonlinear terms, such as $\left(\frac{z}{\rho}\right)^{\neq}, \left(e^H-1-H\right)^{\neq}$, can be controlled  by $z^{\neq}, {\bf w}^{\neq}$ and $H^{\neq}$ due to the effect of operator $(\cdot)^{\neq}$.
To this end, we establish Lemma \ref{L:6.1}. And then we use \eqref{qop} and
Lemma \ref{L:6.1} to obtain
 \begin{align}\label{sa12}
		&\|z^{\neq}, {\bf w}^{\neq}\|_{H^1(\Omega)}^2+\int_0^t\left(\|z^{\neq}\|_{H^1(\Omega)}^2+\|{\bf w}^{\neq}\|_{H^2(\Omega)}^2\right)d\tau
	\lesssim \mathbb{N}^2(0)+\int_0^t\|\nabla H^{\neq}\|_{L^2(\Omega)}^2d\tau.
\end{align} 

To deal with the last term on the right hand side of \eqref{sa12}, we need to use the following structures of perturbation system for the non-zero mode:
\begin{equation}\label{nb1}
\dive (\tilde \rho {\bf w}^{\neq})=-\partial_t z^{\neq} -\dive {\bf J}_1 \quad \mbox{and}\quad z^{\neq}=- \Delta H^{\neq}+e^{\tilde E}H^{\neq}+N_2^{\neq}.
\end{equation}
And then by choosing suitable weighted function and using Lemma \ref{L:6.1} we obtain
 \begin{align}\label{sa11}
		&\|H^{\neq}\|_{H^2(\Omega)}^2+\|H_t^{\neq}\|_{H^1(\Omega)}^2+\int_0^t\left(\|H^{\neq}, H_t^{\neq}\|_{H^2(\Omega)}^2\right)d\tau
	\lesssim \mathbb{N}^2(0)+(\delta+\nu)\| \nabla z^{\neq}, \nabla {\bf w}^{\neq}, \nabla^2{\bf w}^{\neq}\|_{L^2(\Omega)}^2.
\end{align} 
Combining \eqref{sa12} with \eqref{sa11} leads to 
 \begin{align}\label{sa11j}
		&\|H^{\neq}\|_{H^2(\Omega)}^2+\|z^{\neq}, {\bf w}^{\neq}, H_t^{\neq}\|_{H^1(\Omega)}^2+\int_0^t\left(\|z^{\neq}\|_{H^1(\Omega)}^2+\|{\bf w}^{\neq}, H^{\neq}, H_t^{\neq}\|_{H^2(\Omega)}^2\right)d\tau
	\lesssim \mathbb{N}^2(0).
\end{align} 
Similarly, we can further obtain a second-order estimate
 \begin{align}\label{sa13}
		&\|\nabla^2z^{\neq}, \nabla^2{\bf w}^{\neq}, \nabla^3 H^{\neq}, \nabla^2 H_t^{\neq}\|_{L^2(\Omega)}^2+\int_0^t\|\nabla^2z^{\neq}, \nabla^3 {\bf w}^{\neq}, \nabla^3 H^{\neq}, \nabla^3 H_t^{\neq}\|_{L^2(\Omega)}^2d\tau
	\lesssim \mathbb{N}^2(0).
\end{align} 
From \eqref{sa11j}-\eqref{sa13} we get
\begin{align*}
	\|H^{\neq}\|_{H^3(\Omega)}^2+\|H_t^{\neq}, z^{\neq}, {\bf w}^{\neq}\|_{H^2(\Omega)}^2\lesssim C_0e^{-ct},
\end{align*}
together with the fact that $H^{\neq}$ satisfies the Poisson equation \eqref{nb1}, which implies \eqref{E:1.17}.

\

The arrangement of the present paper is as follows. In Section \ref{sp}, we just list some useful lemmas will be needed in what follows.  Section \ref{sw} is devoted to showing the local existence and define {\it a priori} assumption. In Section \ref{s2}, we will establish the estimates of 1-d zero mode $(Z, R, \bar H)(x_1, t)$. Section \ref{s3} is devoted to deriving the estimates of  $(z, {\bf w}, H, H_t)(x, t)$. The exponential decay rate of 3-d non-zero mode $(z^{\neq}, {\bf w}^{\neq}, H^{\neq}, H_t^{\neq})(x, t)$ is obtained in Section \ref{s4}.

\

\noindent{\bf Notations}. \ \ Throughout this paper, the symbol $A\lesssim B$ means that there is a positive constant $C$ such that $A\leq C B$, as well as  the symbol $A\gtrsim B$ means that there is a positive constant $c$ such that $A\geq c B$.
In addition, for any nonnegative multi-index $\alpha=(\alpha_1,\alpha_2, \alpha_3)$ with order $|\alpha|=\alpha_1+\alpha_2+\alpha_3$, define $\partial
_x^\alpha f(x)=\partial
_{x_1}^{\alpha_1}\partial
_{x_2}^{\alpha_2}\partial_{x_3}^{\alpha_3}f(x)$. If $k$ is a nonnegative integer,
we define $\nabla^kf(x) :=\{ \partial^{\alpha}f(x), \forall \,|\alpha| = k\}$ and $|\nabla^kf|=\displaystyle(\sum_{|\alpha| = k} |\partial^{\alpha}f|^2)^{1/2}$.

\section{Preliminaries}\label{sp}
Let $(\tilde \rho, \tilde m_1, \tilde E)(x_1, t):=(\tilde \rho^s, \tilde m_1^s, \tilde E^s)(\xi)$ with $\xi=x_1-st$ be the viscous 2-shock wave solution to the Cauchy problem \eqref{E:1.1c}-\eqref{E:1.3cc} for 1-d NSP. Then we have 
\begin{lemma}[\cite{Duan-Liu-Zhang}]\label{nl}
Let $T \geq 0$. For given data $\left[\rho_{-}, u_{-}\right]$ with $\rho_->0$, there exist positive constants $\delta>0, \bar{C}$ and $\underline{C}$, such that if $\left[\rho_{+}, u_{+}\right]$ satisfies  \eqref{E:1.9}-\eqref{E:1.10} and
$$
|\rho_{-}-\rho_{+}|\leq \delta,
$$
the problem \eqref{E:1.1c}-\eqref{E:1.3cc} has a unique (up to a shift in $\xi$ ) solution $(\tilde \rho^s, \tilde {u}_1^s, \tilde E^s)$ satisfying
\begin{align}\label{e:2.1}
\underline{C} \tilde E^s_{\xi} \leq \tilde \rho^s_{\xi}=\frac{(\tilde \rho^s)^2 \tilde u^s_{1\xi}}{\rho_{-}\left|u_{1-}-s\right|} \leq \bar{C} \tilde E^s_{\xi}<0, \quad \forall \, \xi \in \mathbb{R}.
\end{align}
Moreover, by a suitable choice of the shift, the solution satisfies
$$
\left|\frac{\mathrm{d}^k}{\mathrm{~d} \xi^k}\left[\tilde{\rho}-\rho_{ \pm}, \tilde{u}_1-u_{1 \pm}, \tilde E-E_{ \pm}\right](\xi)\right| \leq C_k\left|\rho_{+}-\rho_{-}\right|^{k+1} e^{-\theta\left|\left(\rho_{+}-\rho_{-}\right) \xi\right|}
$$
for $\xi \lessgtr 0$ and $k=0,1, \cdots$, where each $C_k$ and $\theta$ are generic positive constants.
\end{lemma}

As shown in \cite{Yuan}, the 3-d functions that are periodic in $x^{\prime}=(x_2, x_3)\in \mathbb T^2$ satisfies the following Gagliardo-Nirenberg (G-N) inequalities relevant to the unbounded domain $\Omega=\mathbb R \times \mathbb T^2$, which is different from the G-N inequality in general.
\begin{lemma}[\cite{Yuan}]\label{de}
Assume that $u(x)$ is periodic in $x^{\prime}=\left(x_2, x_3\right)$ and belongs to the $L^q(\Omega)$ space with $\nabla^m u \in L^r(\Omega)$, where $1 \leqslant q, r \leqslant+\infty$ and $m \geqslant 1$. Then there exists a decomposition $\displaystyle u(x)=\sum_{k=1}^3 u^{(k)}(x)$ such that
1) each $u^{(k)}$ satisfies that
$$
\left\|\nabla^j u^{(k)}\right\|_{L^p(\Omega)} \lesssim\left\|\nabla^j u\right\|_{L^p(\Omega)},
$$
where $j \geqslant 0$ is any integer and $p \in[1,+\infty]$ is any number;
2) each $u^{(k)}$ satisfies the $k$-dimensional $G$ - $N$ inequality, namely,
$$
\left\|\nabla^j u^{(k)}\right\|_{L^p(\Omega)} \lesssim\left\|\nabla^m u\right\|_{L^r(\Omega)}^{\theta_k}\|u\|_{L^q(\Omega)}^{1-\theta_k},
$$
where $0 \leqslant j<m$ is any integer and $1 \leqslant p \leqslant+\infty$ is any number, satisfying
$$
\frac{1}{p}=\frac{j}{k}+\left(\frac{1}{r}-\frac{m}{k}\right) \theta_k+\frac{1}{q}\left(1-\theta_k\right) \quad \text { with } \quad \frac{j}{m} \leqslant \theta_k \leqslant 1.
$$
Moreover, it holds that
$$
\left\|\nabla^j u\right\|_{L^p(\Omega)} \lesssim \sum_{k=1}^3\left\|\nabla^m u\right\|_{L^r(\Omega)}^{\theta_k}\|u\|_{L^q(\Omega)}^{1-\theta_k}.
$$
\end{lemma}

For any function $f\in L^{\infty}(\Omega)$ that is periodic in $x^{\prime}=(x_2, x_3)\in \mathbb T^2$, the zero mode $\bar f$ and non-zero mode $f^{\neq}$ defined in \eqref{E:1.18} possess the following properties.
\begin{lemma}[\cite{Yuan}]\label{l1}
For any $p \in[1,+\infty]$, it holds that
\begin{align*}
\left\|\bar f\right\|_{L^p(\mathbb{R})} & \lesssim\|f\|_{L^p(\Omega)}, 
\notag\\
\left\|f^{\neq}\right\|_{L^p(\Omega)} & \lesssim\|f\|_{L^p(\Omega)}+\left\|\bar f\right\|_{L^p(\mathbb{R})} \lesssim\|f\|_{L^p(\Omega)}
\end{align*}
and
\begin{align*}
\left\|f^{\neq}\right\|_{L^p(\Omega)} \lesssim\left\|\nabla_{x^{\prime}} f^{\neq}\right\|_{L^p(\Omega)}. 
\end{align*}
\end{lemma}

{\bf Galilean transformation.} For any constant $a \in \mathbb{R}$, one has 
 the Galilean transformation,
$$
\left(\rho^*, u_1^*, u_2^*, u_3^*, E^*\right)\left(x^*, t^*\right)=\left(\rho, u_1-a, u_2, u_3, E\right)\left(x_1^*+a t^*, x_2^*, x_3^*, t^*\right), \quad x^* \in \mathbb{R}^3, t^* \geq 0.
$$
Then the compressible NSP system \eqref{E:1.1}, the Rankine-Hugoniot conditions \eqref{E:1.9} and the entropy condition \eqref{E:1.10} are invariant under the Galilean transformation. Note that the transformed shock wave  $\left(\tilde{\rho}^{*}, \tilde{u}_1^{*}, \tilde{E}^{*}\right)\left(x_1^*, t^*\right)$ propagates along the $x_1^*$-axis with the shock speed $s^*:=s-a$ and connects the shock states $\left({\rho}_{ \pm}^*, {u}_{1\pm}^*, {E}_{ \pm}^*\right):=\left({\rho}_{ \pm}, {u}_{1\pm}-a, {E}_{ \pm}\right)$ as $x_1^* \rightarrow \pm \infty$. Then ${u}_{1-}^*=-{u}_{1+}^*$ follows as $a=\frac{{u}_{1-}+{u}_{1+}}{2}$.
Thus, without loss of generality we assume 
$$
{u}_{1-}=-{u}_{1+}>0
$$
in this paper. Then it follows from \eqref{E:1.9} that 
$$
{u}_{1-}=\left|{u}_{1+}\right| \lesssim \left|{\rho}_{+}-{\rho}_{-}\right|\leq \delta,
$$
together with $\tilde u^s_{1\xi}<0$ from \eqref{e:2.1}, which yields that
\begin{align}\label{cs}
\left|\tilde{u}_1(x_1, t)\right|=\left|\tilde{u}_1^s\left(\xi\right)\right| \leq \left|{u}_{1 \pm}\right| \lesssim \delta, \quad \forall \,x_1 \in \mathbb{R}, \, t>0.
\end{align}

\section{Reformulated problem}\label{sw}
It follows from \eqref{E:1.1c} that the planar viscous shock wave $(\tilde \rho,\tilde {\bf m},\tilde E)(x,t)$ satisfies the following system:
\begin{align}\label{E:1.2} 
	\left\{\begin{array}{l}
		\partial_t \tilde{\rho}+\operatorname{div} \tilde{\bf  m}=0, \\
		\partial_t \tilde{\bf m}+\operatorname{div}\left(\frac{\tilde{\bf m} \otimes \tilde{\bf m}}{\tilde{\rho}}\right)+T\nabla\tilde{\rho}=\Delta \tilde{{\bf u}}-\tilde{\rho} \nabla \tilde{E}, \\
		-\Delta \tilde{E}=\tilde{\rho}-e^{\tilde{E}}
	\end{array}\right.
\end{align}
with 
\begin{align}\label{E:1.2c}
	\begin{cases}
		\tilde \rho(x,0)=\tilde \rho_0(x_1)\rightarrow \rho_{ \pm}>0 \quad \mbox{as} \quad x_1\rightarrow \pm \infty,\\
		\tilde {\bf m}(x,0)=(\tilde m_{10}(x_1), 0, 0)^T\rightarrow (m_{1\pm}, 0, 0)^T \quad \mbox{as} \quad x_1\rightarrow \pm \infty
	\end{cases}
\end{align}
and 
\begin{equation}\label{bE:1.3c}
	\lim\limits_{x\rightarrow \pm \infty}\tilde E(x_1, t)=E_{\pm}.
\end{equation}
Set $z=\rho-\tilde {\rho}, {\bf w}={\bf u}-\tilde{{\bf u}}, H=E-\tilde{E}$. Then by \eqref{E:1.1}-\eqref{E:1.12b} and \eqref{E:1.2}-\eqref{bE:1.3c}, $(z,{\bf w}, H)(x, t)$ admits 
\begin{align}\label{E:5.3}
	\begin{cases}
		\partial_t z+\dive (\rho {\bf w}+z\tilde {\bf u})=0,\\
		\partial_t{\bf w}+\frac{T}{\rho}\nabla z-\frac{1}{\rho}\Delta {\bf w}+\nabla H=-{\bf h_1}-{\bf h_2},\\
		-\Delta H=z-N_1=z-e^{\tilde E}H-N_2,
	\end{cases}
\end{align}
with 
\begin{align}
	\begin{cases}
		z(x,0)=z_0(x)=(\rho_0-\tilde \rho_0)(x)\rightarrow 0 \quad \mbox{as} \quad x_1\rightarrow \pm \infty,\\
		{\bf w}(x,0)={\bf w}_0(x)=(u_{10}-\tilde u_{10}, u_{20}, u_{30})^T(x)\rightarrow {\bf 0} \quad \mbox{as} \quad x_1\rightarrow \pm \infty
	\end{cases}
\end{align}
and 
\begin{equation}\label{E:1.3c}
	\lim\limits_{x\rightarrow \pm \infty}H(x, t)=0,
\end{equation}
where 
\begin{align*}
	{\bf h_1}&={\bf u}\cdot \nabla {\bf w}=\tilde {\bf u}\cdot \nabla {\bf w}+{\bf w}\cdot \nabla {\bf w},\notag\\
	{\bf h_2}&=\left(\frac{1}{\rho}-\frac{1}{\tilde\rho}\right)\left(T\nabla \tilde \rho-\Delta \tilde {\bf u}\right)+\nabla \tilde {\bf u}\cdot {\bf w}=\left(\frac{1}{\rho}-\frac{1}{\tilde\rho}\right)\left(T\nabla \tilde \rho-\Delta \tilde {\bf u}\right)+w_1\partial_1 \tilde u_1,\notag\\
	N_1&=e^{\tilde E}(e^H-1), \quad N_2=e^{\tilde E}(e^H-1-H).
\end{align*}

It is noted that the potential term $\rho \nabla E$  can be converted into 
\begin{align*}
	\rho \nabla E 
	& =-\operatorname{div}(\nabla E \nabla E)+\frac{1}{2} \nabla(\nabla E)^2+\nabla \rho+ \nabla(\Delta E)
\end{align*}
and $\tilde \rho \nabla \tilde E$ has the same form.
Set the perturbation of momentum ${\bf r}=\left(r_1, r_2, r_3\right)^T={\bf m}-\tilde{{\bf m}}$, then from \eqref{E:1.1} and \eqref{E:1.2} we obtain the perturbation system of $(z, {\bf r}, H)(x,t)$ as follows:
%
\begin{align}\label{E:1.6}
\left\{\begin{array}{l}
	\partial_t z+\operatorname{div} {\bf r}=0, \\
	\partial_t {\bf r}+\operatorname{div}\left(\frac{{\bf m} \otimes {\bf m}}{\rho}-\frac{\tilde{\bf m} \otimes \tilde{{\bf m}}}{\tilde{\rho}}\right)+(T+1) \nabla z+\nabla\Delta H=\Delta {\bf w}+{\bf L}, \\
	-\Delta H=z-N_1=z-e^{\tilde E}H-N_2,
\end{array}\right.
\end{align}
with the initial data $z_0(x)$ and ${\bf r}_0(x)$,
where 
\begin{align}
	&{\bf L}=\operatorname{div}(\nabla E \nabla E-\nabla \tilde{E} \nabla \tilde{E})-\frac{1}{2}\nabla\left((\nabla E)^2-(\nabla \tilde{E})^2\right).\nonumber
\end{align}
It follows from \eqref{E:1.6} that
$$
\begin{array}{ll}
	\frac{d}{d t} \left(\int_{\Omega} z d x\right)=0 \quad \mbox{and}\quad  \frac{d}{d t} \left(\int_{\Omega} {\bf r} d x\right)={\bf 0},
\end{array}
$$
which yields that 
$$
\begin{array}{ll}
	\int_{\Omega} z(x, t) d x=\int_{\Omega} z(x, 0) d x\quad \mbox{and}\quad\int_{\Omega} {\bf r}(x, t) d x=\int_{\Omega} {\bf r}(x, 0) d x .
\end{array}
$$
Under the zero-mass type condition \eqref{zc}
 we have
$$
\int_{\Omega} z(x, t) d x=0 \quad \text { and }\quad \int_{\Omega} {\bf r}(x, t) d x={\bf 0}.
$$
Integrating \eqref{E:1.6} with respect to $x^{\prime}$ over $\mathbb{T}^2$ gives that
\begin{align*}
	\left\{\begin{array}{l}
		\partial_t \bar z+\partial_1 \bar r_1=0, \\
		\partial_t {\bf r}^{od}+\partial_1 \int_{\mathbb{T}^2}\left(\frac{{\bf m} m_1}{\rho}-\frac{\tilde{\bf m} \tilde{m}_1}{\tilde{\rho}}\right) d x^{\prime}+(T+1) \partial_1 \bar z {\bf e}_1+\partial_1^3\bar H{\bf e}_1=\partial_1^2 \bar{\bf w}+\bar {\bf L}, \\
		-\partial_1^2 \bar H=\bar z-e^{\tilde E}\bar H-\overline{N_2},
	\end{array}\right.
\end{align*}
where $\bar f$ is defined in \eqref{E:1.18} and ${\bf e}_1=(1,0,0)^T$.
Thus, $(\bar z, \bar r_1, \bar H)(x, t)$ satisfies
\begin{align}\label{E:2.2}
	\left\{\begin{array}{l}
		\partial_t \bar z+\partial_1 \bar r_1=0, \\
		\partial_t \bar r_1+\partial_1 \int_{\mathbb{T}^2}\left(\frac{m_1^2}{\rho}-\frac{ \tilde{m}_1^2}{\tilde{\rho}}\right) d x^{\prime}+(T+1) \partial_1 \bar z +\partial_1^3\bar H=\partial_1^2 w_1^{od}+\partial_1\overline{L_1}, \\
		-\partial_1^2 \bar H=\bar z-e^{\tilde E}\bar H-\overline{N_2},
	\end{array}\right.
\end{align}
where 
\begin{align}
	\overline{L_1}=&\frac{1}{2}\left[\int_{\mathbb{T}^2}\left((\partial_1E)^2-(\partial_1\tilde E)^2\right)dx^{\prime}-\int_{\mathbb{T}^2}\left((\partial_2H)^2+(\partial_3 H)^2\right)dx^{\prime}\right],\label{E:4.7-2}\\
	\overline{N_2}=&e^{\tilde E}\int_{\mathbb{T}^2}\left(e^H-1-H\right)dx^{\prime}.\label{E:4.7}
\end{align}
Let
\begin{align}
	& Z=\int_{-\infty}^{x_1} \int_{\mathbb{T}^2} z\left(y_1, x^{\prime}, t\right) d x^{\prime} d y_1=\int_{-\infty}^{x_1} \bar z\left(y_1, t\right) d y_1, \label{i1}\\
	& R=\int_{-\infty}^{x_1} \int_{\mathbb{T}^2} r_1\left(y_1, x^{\prime}, t\right) d x^{\prime} d y_1=\int_{-\infty}^{x_1} \bar r_1\left(y_1, t\right) d y_1. \label{i2}
\end{align}
Integrating $\eqref{E:2.2}_1$ and $\eqref{E:2.2}_2$ from $-\infty$ to $x_1$ gives that
\begin{align}\label{E:4.8}
	& \left\{\begin{array}{l}
		\partial_t Z+\partial_1 R=0, \\
		\partial_t R+2 \tilde{u}_1 \partial_1  R+\left(T+1-\tilde{u}_1^2\right) \partial_1 Z-\partial_1\left[\frac{1}{\tilde{\rho}}\left(\partial_1 R-\tilde{u}_1 \partial_1 Z\right)\right]+\partial_1^2 \bar H=\partial_1 f_2-f_1+\overline{L_1}, \\
		-\partial_1^2 \bar H=\bar z-e^{\tilde E}\bar H-\overline{N_2}
	\end{array}\right. 
\end{align}
with the initial data 
\begin{align}
(Z, R)(x_1, 0)=(Z_0, R_0)(x_1):=\int_{-\infty}^{x_1}(\bar z_0, \bar r_{10})(y_1)dy_1,\nonumber
\end{align}
where
\begin{align}
	&f_1=\int_{\mathbb{T}^2}\left(\frac{m_1^2}{\rho}-\frac{\tilde{m}_1^2}{\tilde{\rho}}-\frac{2\tilde m_1}{\tilde{\rho}}\left(m_1-\tilde{m}_1\right)+\frac{\tilde{m}_1^2}{\tilde{\rho}^2}(\rho-\tilde{\rho})\right) d x^{\prime}, \label{E:4.9}\\
&f_2=\int_{\mathbb{T}^2}\left(\frac{m_1}{\rho}-\frac{\tilde{m}_1}{\tilde{\rho}}-\frac{1}{\tilde{\rho}}\left(m_1-\tilde{m}_1\right)+\frac{\tilde{m}_1}{\tilde{\rho}^2}(\rho-\tilde{\rho})\right) d x^{\prime}.\label{E:4.10}
\end{align}

Motivated by \cite{Matsumura-Nishida}, we seek for the solution of \eqref{E:5.3}-\eqref{E:1.3c} in the following solution space
\begin{align*}
X_{t_0}=\Big\{(z, {\bf w}, H)(x,t) \Big|&(z, {\bf w}, H_t)\in C\left(0, t_0 ; H^4(\Omega)\right),   \,H \in C\left(0, t_0 ; H^5(\Omega)\right),
\notag\\
 &\nabla z \in L^2\left(0, t_0 ; H^3(\Omega)\right), \nabla {\bf w} \in L^2\left(0, t_0 ; H^4(\Omega)\right), \nabla H \in L^2\left(0, t_0 ; H^4(\Omega)\right), 
 \notag\\
 &\nabla H_t \in L^2\left(0, t_0 ; H^4(\Omega)\right) \quad\mbox{and}\quad \mathbb{N}_1(t_0)\leq E_0\Big\}
\end{align*}
for some positive constant $E_0$, where 
\begin{align*}
\mathbb{N}_1(t_0):=\sup_{t \in[0, t_0]}\left(\|z, {\bf w}, H_t\|_{H^4(\Omega)}+\|H\|_{H^5(\Omega)}\right).
\end{align*}

\begin{theorem}[Local existence]\label{le}
Under the assumptions of Theorem \ref{theorem1.1}, there exist $t_0>0, \delta_0>0$ and $\nu_0>0$ such that if
\begin{align}\label{u1}
\delta \leqslant \delta_0, \quad \left\|z_0, {\bf r}_0\right\|_{H^4(\Omega)}+\left\|Z_0, R_0\right\|_{L^2(\mathbb{R})} \leqslant \nu_0,
\end{align}
then the problem \eqref{E:5.3}-\eqref{E:1.3c} admits a unique solution $(z, {\bf w}, H)(x,t) \in X_{t_0}$. Moreover, the anti-derivative $(Z, R)(x_1, t)$ given as in \eqref{i1}-\eqref{i2} exists and $(Z, R)(x_1, t) \in C\left(0, t_0 ; H^5(\mathbb{R})\right)$ with 
\[\mathbb{N}_2(t_0):=\sup_{t \in[0, t_0]}\left\|Z, R\right\|_{L^2(\mathbb{R})}\leq E_1\]
for some positive constant $E_1$.
\end{theorem}

\begin{proof}
It is standard (see \cite{Matsumura-Nishida}) to prove that there exists some $t_0>0$ such that if the initial data satisfies \eqref{u1}, then the problem \eqref{E:5.3}-\eqref{E:1.3c} admits a local solution $(z, {\bf w}, H)(x,t) \in X_{t_0}$, which is periodic in the transverse directions $x^{\prime} \in \mathbb T^2$. 
Then it follows from ${\bf r}={\bf m}-\tilde{\bf m}=\rho {\bf w}+\tilde{\bf u}z$ that $(z, {\bf r}, H)(x,t) \in X_{t_0}$ is the smooth solution to the corresponding Cauchy problem of \eqref{E:1.6}. Thus, it suffices to show that the existence of anti-derivative $(Z, R) \in C\left(0, t_0 ; H^5(\mathbb{R})\right)$ with $N_2(t_0)\leq E_1$.

We rewrite the first two equations of \eqref{E:2.2} as 
\begin{align}\label{i3}
	\left\{\begin{array}{l}
		\partial_t \bar z+\partial_1 \bar r_1=0, \\
		\partial_t \bar r_1-\partial_1\left(\frac{\partial_1\bar r_1}{\tilde \rho}\right)=\partial_1K,
	\end{array}\right.
\end{align}
where 
\begin{align*}
		K=\int_{\mathbb T^2}\left(\partial_1w_1-\frac{\partial_1r_1}{\tilde \rho}\right)dx^{\prime}- \int_{\mathbb{T}^2}\left(\frac{m_1^2}{\rho}-\frac{ \tilde{m}_1^2}{\tilde{\rho}}\right) d x^{\prime}-(T+1) \bar z -\partial_1^2\bar H+\overline{L_1}.
\end{align*}
And from Lemma \ref{l1} and $\eqref{E:2.2}_3$ we get $\bar H \in C\left(0, t_0 ; H^6(\mathbb R)\right)$, which yields that $K\in C\left(0, t_0 ; H^3(\mathbb R)\right)$. Then following the framework of Section 6 in \cite{Yuan}, we can verify that $R \in C\left(0, t_0 ; H^4(\mathbb{R})\right)$ is the solution to a Cauchy problem for a linear parabolic equation and satisfies $\partial_1R= \bar r_1$, as well as $Z \in C\left(0, t_0 ; H^3(\mathbb{R})\right)$ is the solution to a Cauchy problem for a linear hyperbolic equation and $\partial_1Z=\bar z$,
where we omit the details. Thus it follows from $(z, {\bf w}, H)(x,t) \in X_{t_0}$ that $(Z,R) \in C\left(0, t_0 ; H^5(\mathbb{R})\right)$.
Furthermore, it follows from the classical parabolic theory that $\mathbb{N}_2(t_0)\leq E_1$ for some positive constant $E_1$. Thus, the proof is completed.
\end{proof}

Since the local existence of the
solution of \eqref{E:5.3}-\eqref{E:1.3c} has been shown in Theorem \ref{le}, it remains to establish
the {\it a priori} estimates for the solution in what follows.
For any $T\in (0,+\infty)$, define
\begin{align}\label{ass}
\mathbb{N}(T):=\sup_{t \in[0, T]}\left(\left\|Z, R\right\|_{L^2(\mathbb{R})}+\|z, {\bf w}, H_t\|_{H^4(\Omega)}+\|H\|_{H^5(\Omega)}\right)=\nu\ll 1,
\end{align}
then it follows from Lemma \ref{de} that
\begin{align}\label{E:3.1}
	 \sup_{t \in[0, T]}\|z, {\bf w}, H, H_t\|_{w^{2, \infty}(\Omega)} &\lesssim \sup_{t \in[0, T]}\left\{\|\nabla(z, {\bf w}, H, H_t)\|_{H^{2}(\Omega)}^{\frac{1}{2}}\|z, {\bf w}, H, H_t\|_{H^{2}(\Omega)}^{\frac{1}{2}}+\|\nabla(z, {\bf w}, H, H_t)\|_{H^{2}(\Omega)}\right. \notag\\
	&\quad\quad \quad\quad\quad\,\, \left.+\left\|\nabla^2(z, {\bf w}, H, H_t)\right\|_{H^{2}(\Omega)}^{\frac{3}{4}}\|z, {\bf w}, H, H_t\|_{H^{2}(\Omega)}^{\frac{1}{4}}\right\}
	\notag\\
	&\lesssim \nu.
\end{align}
Note that ${\bf r}={\bf m}-\tilde{{\bf m}}=\rho {\bf u}-\tilde{\rho} \tilde{{\bf u}}=\rho{\bf w}+z \tilde{{\bf u}}$, which leads to
$$
\sup_{t \in[0, T]}\|{\bf r}\|_{w^{2, \infty}(\Omega)} \lesssim \sup_{t \in[0, T]}\|z, {\bf w}\|_{w^{2, \infty}(\Omega)} \lesssim \nu.
$$
Combining Sobolev inequality and Lemma \ref{l1}, one get
\begin{align}
	\sup_{t \in[0, T]}\|Z, R, \bar H\|_{w^{4, \infty}(\mathbb{R})} &\lesssim \sup_{t \in(0, T)}\|Z, R, \bar H\|_{H^{5}(\mathbb{R})}
	\notag\\
	&\lesssim \sup_{t \in[0, T]}\left(\|Z, R, \bar H\|_{L^{2}(\mathbb{R})}+\|\bar z, \bar r_1, \partial_1 \bar H\|_{H^{4}(\mathbb{R})}\right)
	\notag\\
	&\lesssim \sup_{t \in[0, T]}\left(\|Z, R\|_{L^{2}(\mathbb{R})}+\|H\|_{L^{2}(\Omega)}+\|z, r_1, \partial_1 H\|_{H^{4}(\Omega)}\right)\lesssim \nu,
\end{align}
where we have used the fact $\partial_1 \bar H=\overline{\partial_1 H}$.

\section{the estimates of 1-d zero mode}\label{s2}
This section is devoted to establishing the estimates of 1-d zero mode $(Z, R, \bar H)(x_1, t)$ of the solution $(z, {\bf w}, H)(x, t)$ to the problem \eqref{E:5.3}-\eqref{E:1.3c} under the {\it a priori} assumption \eqref{ass}.
Recall that \eqref{cs}, one get 
$$
\alpha:=T+1-\tilde u_1^2\geq \frac12(T+1)>0.
$$
\begin{lemma}\label{L:4.1}
For any $T>0$, assume that $(z, {\bf w}, H)(x,t) \in X_{T}$ is the solution to the problem \eqref{E:5.3}-\eqref{E:1.3c}. If $\delta$ and $\nu$ are small, then it holds that
	\begin{align}\label{E:4.11}
		&\|Z,R\|_{L^2(\mathbb{R})}^2+\|\bar H\|_{H^1(\mathbb{R})}^2+\int_0^t \|\sqrt{-\partial_1\tilde u_1}R, \partial_1R\|_{L^2(\mathbb{R})}^2d\tau
		\notag\\
		\lesssim &\mathbb{N}^2(0)+\nu \|\nabla H\|_{L^2(\Omega)}^2
+(\delta+\nu) \int_0^t\left(\left\|\partial_1 Z\right\|_{L^2(\mathbb{R})}^2+\left\|\bar H\right\|_{H^{1}(\mathbb{R})}^2+\|z, r_1, H_t\|_{L^2(\Omega)}^2+\left\|H\right\|_{H^1(\Omega)}^2\right) d \tau. 
\end{align}
\end{lemma}
\begin{proof}
	Multiplying $\eqref{E:4.8}_1$ and $\eqref{E:4.8}_2$ by $Z$ and $\frac{R}{\alpha}$ respectively and adding them up, we have
\begin{align}\label{E:4.15}
	&  \partial_t\left(\frac{Z^2}{2}+\frac{R^2}{2\alpha}\right)  +\beta R^2+\frac{1}{\alpha \tilde \rho}(\partial_1R)^2
	 -\frac{\partial_1 \bar H \partial_1 R}{T+1}
	 \notag\\
	 =&\partial_1(\cdots)+\underbrace{\frac{\partial_1\alpha}{\tilde{\rho}\alpha^2}R\left(\partial_1 R-\tilde{u}_1 \partial_1 Z\right)+\frac{\tilde u_1}{\alpha \tilde \rho}\partial_1 Z\partial_1 R-\partial_1 \bar H\left[\partial_1 R\left(\frac{1}{\alpha}-\frac{1}{T+1}\right)+R\partial_1\left(\frac{1}{\alpha}\right)\right]}_{I_1}
	 \notag\\
	 &\underbrace{-\left(f_2 \partial_1\left(\frac{R}{\alpha}\right)+f_1 \frac{R}{\alpha}\right)}_{I_2}+\overline{L_1} \frac{R}{\alpha},
\end{align}
where $(\cdots)=-Z R-\frac{\tilde{u}_1}{\alpha} R^2+\frac{1}{\tilde{\rho}}\left(\partial_1 R-\tilde{u}_1 \partial_1 Z\right)\frac{R}{\alpha}-\frac{R\partial_1 \bar H}{\alpha}+f_2\frac{R}{\alpha}$ and 
$
\beta=-\left(\frac{1}{2 \alpha}\right)_t-\partial_1\left(\frac{\tilde u_1}{\alpha}\right).
$
It follows from $\eqref{E:4.8}_1$ and $\eqref{E:4.8}_3$ that
\begin{align}\label{E:4.16}
-\frac{\partial_1 \bar H \partial_1 R}{T+1}
=&\left(\frac{\partial_1 \bar H Z}{T+1}\right)_t-\partial_1\left(\frac{\bar H_t Z}{T+1}\right)+\frac{\bar H_t }{T+1}\bar z
\notag\\
=&\left(\frac{\partial_1 \bar H Z}{T+1}\right)_t-\partial_1\left(\frac{\bar H_t Z}{T+1}\right)+\frac{\bar H_t }{T+1}\left(- \partial_1^2 \bar H+e^{\tilde{E}} \bar H+\overline{N_2}\right)
\notag\\
=&\left(\frac{\partial_1 \bar H Z}{T+1}+\frac{e^{\tilde E}(\bar H)^2 }{2(T+1)}+\frac{(\partial_1 \bar H)^2 }{2(T+1)}+\frac{\bar H\overline{N_2}}{T+1}\right)_t+\partial_1(\cdots)-\underbrace{\frac{e^{\tilde E}\tilde E_t(\bar H)^2 }{2(T+1)}}_{I_3}-\underbrace{\frac{\bar HN_{2t}^{od}}{T+1}}_{I_4},
\end{align}
where 
$(\cdots)=-\frac{\bar H_t Z}{T+1}-\frac{\bar H_t\partial_1\bar H}{T+1}$. Substituting \eqref{E:4.16} into \eqref{E:4.15} and integrating the resultant over $\mathbb{R}$ gives that
\begin{align}\label{E:4.19}
	&\quad \frac{d}{dt}\left[\int_{\mathbb{R}}\left(\frac{Z^2}{2}+\frac{R^2}{2 \alpha}+\frac{\partial_1 \bar H Z}{T+1}+\frac{e^{\tilde E}\left(\bar H\right)^2}{2(T+1)}+\frac{\left(\partial_1 \bar H\right)^2}{2(T+1)}+\frac{\bar H \overline{N_2}}{T+1}\right) d x_1\right]
	\notag\\
	&+\int_{\mathbb{R}}\beta R^2dx_1+\int_{\mathbb{R}}\frac{1}{\alpha \tilde \rho}(\partial_1R)^2dx_1\leq \sum_{i=1}^{4}\|I_i\|_{L^1(\mathbb{R})}+\Big\|\frac{\overline{L_1}R}{\alpha}\Big\|_{L^1({\mathbb{R}})}.
\end{align}
It is worth to pointing out that 
\begin{align}
\beta=\frac{\alpha_t}{2\alpha^2}+\frac{\tilde u_1\partial_1\alpha}{\alpha^2}-\frac{\partial_1\tilde u_1}{\alpha}=\frac{-\partial_1\tilde u_1}{\alpha}\left[1+\frac{\tilde u_1(2\tilde u_1-s)}{\alpha}\right]\gtrsim -\partial_1\tilde u_1>0.\nonumber
\end{align}

We now estimate each $I_i$ from $i=1$ to $4$. Firstly, by the smallness of $\|\tilde u_1\|_{W^{1,\infty}(\mathbb R)}$ in Lemma \ref{nl} we have
\begin{align}
\|I_1\|_{L^1(\mathbb{R})}+	\|I_3\|_{L^1(\mathbb{R})}
\lesssim \delta \|\sqrt{-\partial_1\tilde u_1}R, \partial_1R, \partial_1Z, \bar H, \partial_1\bar H\|_{L^2(\mathbb{R})}^2.\label{E:4.17}
\end{align}	
It follows from \eqref{E:4.7-2} and \eqref{E:4.9}-\eqref{E:4.10} that 
	\begin{align}\label{E:4.21}
		\|I_2\|_{L^1(\mathbb{R})}
		\leq &\int_{\mathbb{R}}|f_2|\Big|\partial_1\left(\frac{R}{\alpha}\right)\Big|dx_1+\int_{\mathbb{R}}\frac{|f_1R|}{\alpha}dx_1
		\lesssim \|R\|_{W^{1,\infty}(\mathbb{R})}\|z,r_1\|_{L^2(\Omega)}^2	\lesssim \nu\|z,r_1\|_{L^2(\Omega)}^2
	\end{align}	
	and
		\begin{align}
	\Big\|\frac{\overline{L_1}R}{\alpha}\Big\|_{L^1({\mathbb{R}})}=&\int_{\mathbb{R}}\frac{1}{2}\Big|\left[\int_{\mathbb{T}^2}\left((\partial_1E)^2-(\partial_1\tilde E)^2\right)dx^{\prime}-\int_{\mathbb{T}^2}\left((\partial_2H)^2+(\partial_3 H)^2\right)dx^{\prime}\right]\frac{R}{\alpha}\Big|dx_1
	\notag\\
	\lesssim& \|R\|_{L^{\infty}(\mathbb{R})}\|\nabla H\|_{L^2(\Omega)}^2+\int_{\mathbb{R}}\int_{\mathbb{T}^2}|\partial_1H|dx^{\prime}|\partial_1\tilde E R|dx_1
		\notag\\
	\lesssim& \|R\|_{L^{\infty}(\mathbb{R})}\|\nabla H\|_{L^2(\Omega)}^2+\int_{\mathbb{R}}\int_{\mathbb{T}^2}|\partial_1H|dx^{\prime}|\partial_1\tilde u_1 R|dx_1
			\notag\\
	\lesssim& \|R\|_{L^{\infty}(\mathbb{R})}\|\nabla H\|_{L^2(\Omega)}^2+\epsilon\|\sqrt{|\partial_1 \tilde u_1|}R\|_{L^2(\mathbb{R})}^2+|\partial_1 \tilde u_1|_{L^{\infty}}\|\partial_1 H\|_{L^2(\Omega)}^2
	\notag\\
		\lesssim& \epsilon\|\sqrt{|\partial_1 \tilde u_1|}R\|_{L^2(\mathbb{R})}^2+(\delta+\nu)\|\nabla H\|_{L^2(\Omega)}^2,
	\end{align}		
	where $\epsilon>0$ is a small constant. In addition, from \eqref{E:4.7} we have
	\begin{align}\label{E:4.22}
\|I_4\|_{L^1(\mathbb{R})}\lesssim & \int_{\mathbb{R}}|\bar H|\left[\Big|(e^{\tilde E})_t\Big|\int_{\mathbb{T}^2}(e^H-1-H)dx^{\prime}+e^{\tilde E}\int_{\mathbb{T}^2}|(e^H-1)H_t|dx^{\prime}\right]dx_1
\notag\\
\lesssim &\|\bar H\|_{L^{\infty}(\mathbb{R})}\left(\| H\|_{L^2(\Omega)}^2+\| H_t\|_{L^2(\Omega)}^2\right)
\lesssim \nu \| H, H_t\|_{L^2(\Omega)}^2.
\end{align}	
Substituting \eqref{E:4.17}-\eqref{E:4.22} into \eqref{E:4.19} and integrating the results over $[0,t]$ gives that 
	\begin{align}\label{E:4.12}
	& \int_{\mathbb{R}}\left(\frac{Z^2}{2}+\frac{R^2}{2 \alpha}+\frac{\partial_1 \bar H Z}{T+1}+\frac{e^{\tilde E}\left(\bar H\right)^2}{2(T+1)}+\frac{\left(\partial_1 \bar H\right)^2}{2(T+1)}+\frac{\bar H \overline{N_2}}{T+1}\right) d x_1 +\int_0^t \|\sqrt{-\partial_1\tilde u_1}R, \partial_1R\|_{L^2(\mathbb{R})}^2d\tau\notag\\
	\lesssim& \mathbb{N}^2(0)+(\delta+\nu) \int_0^t\left(\left\|\partial_1 Z\right\|_{L^2(\mathbb{R})}^2+\left\|\bar H\right\|_{H^{1}(\mathbb{R})}^2+\|z, r_1, H_t\|_{L^2(\Omega)}^2+\left\|H\right\|_{H^1(\Omega)}^2\right) d \tau,
\end{align}
where we have used the smallness of $\delta$ and $\nu$.
It follows from 
$2^2-4(T+1)<0$ that
\begin{equation}\label{E:4.13}
\int_{\mathbb{R}}\left(\frac{Z^2}{2}+\frac{\partial_1 \bar H Z}{T+1}+\frac{1}{2(T+1)}\left(\partial_1 \bar H\right)^2\right) d x_1\gtrsim \left\| Z\right\|_{L^2(\mathbb{R})}^2+\left\|\partial_1\bar H\right\|_{L^{2}(\mathbb{R})}^2
\end{equation}
and from \eqref{E:3.1} that
\begin{align}\label{E:4.14}
	\int_{\mathbb{R}}\frac{\bar H \overline{N_2}}{T+1} d x_1
	\lesssim \|\bar H\|_{L^{\infty}(\mathbb{R})}\|H\|_{L^2(\Omega)}^2	\lesssim \nu\|H\|_{L^2(\Omega)}^2.
\end{align}
Substituting \eqref{E:4.13}-\eqref{E:4.14} into \eqref{E:4.12} yields \eqref{E:4.11} directly. Thus the proof is completed.
\end{proof}

\begin{lemma}
For any $T>0$, assume that $(z, {\bf w}, H)(x,t) \in X_{T}$ is the solution to the problem \eqref{E:5.3}-\eqref{E:1.3c}. If $\delta$ and $\nu$ are small, then it holds that
\begin{align}\label{E:4.33}
	&\int_{\mathbb{R}}\left(R\partial_1Z+\frac{(\partial_1Z)^2}{2\tilde\rho}\right)dx_1+\int_0^t\|\partial_1Z\|_{L^2(\mathbb{R})}^2d\tau+\int_0^t\|\partial_1\bar H\|_{H^1(\mathbb{R})}^2d\tau
	\notag\\
	\lesssim&\mathbb{N}^2(0)+\int_0^t\|\partial_1R\|_{L^2(\mathbb{R})}^2d\tau+(\delta+\nu)\int_0^t(\|H\|_{H^1(\Omega)}^2+\|\bar H\|_{L^2(\mathbb{R})}^2)d\tau+\nu\int_0^t\|z, r_1\|_{L^2(\Omega)}^2d\tau.
	\end{align}
\end{lemma}
\begin{proof}
By multiplying $\eqref{E:4.8}_2$ and $\eqref{E:4.8}_3$ by $\partial_1Z$ and $-\partial_1^2\bar H$ respectively and adding them up, we have
	\begin{align}\label{E:4.28}
	&\partial_t\left(R\partial_1Z+\frac{1}{2\tilde\rho}\left(\partial_1Z\right)^2\right)+\alpha\left(\partial_1Z\right)^2+\left(\partial_1^2\bar H\right)^2+2\partial_1^2\bar H\partial_1Z+e^{\tilde E}\left(\partial_1\bar H\right)^2
	\notag\\
	=&\partial_1(\cdots)+\left(\partial_1R\right)^2\underbrace{-\left(2\tilde u_1+\frac{\partial_1 \tilde{\rho}}{\tilde {\rho}^2}\right) \partial_1 R \partial_1 Z-\frac{1}{2}\left[\partial_1\left(\frac{\tilde u_1}{\tilde{\rho}}\right)-\partial_t\left(\frac{1}{\tilde  \rho}\right)\right]\left(\partial_1 Z\right)^2-\partial_1(e^{\tilde E})\bar H\partial_1\bar H}_{I_5}
	\notag\\
	&-\partial_1\bar H\partial_1\overline{N_2}-(f_2\partial_1^2Z+f_1\partial_1Z)+\overline{L_1}\partial_1Z,
	\end{align}
	where
	$(\cdots)=-R\partial_1R-\frac{\tilde u_1}{2 \tilde{\rho}} (\partial_1 Z)^2+e^{\tilde E}\bar H\partial_1\bar H+\overline{N_2}\partial_1\bar H+f_2\partial_1Z$ and we have used the fact from $\eqref{E:4.8}_1$ 
	that 
	\begin{align}
\partial_tR\partial_1Z=\partial_t\left(R\partial_1Z\right)-\left(\partial_1R\right)^2+\partial_1\left(R\partial_1R\right).\nonumber
\end{align}
	It follows from $2^2-4\alpha<0$ that 
	\begin{align}\label{E:4.31}
	\alpha\left(\partial_1Z\right)^2+\left(\partial_1^2\bar H\right)^2+2\partial_1^2\bar H\partial_1Z\gtrsim \left(\partial_1Z\right)^2+\left(\partial_1^2\bar H\right)^2.
	\end{align}
	The direct computations shows that 
	\begin{align}
	\|I_5\|_{L^1(\mathbb{R})}\lesssim \delta \left(\|\partial_1Z, \partial_1R\|_{L^{2}(\mathbb{R})}^2+\|\bar H\|_{H^{1}(\mathbb{R})}^2\right)
	\end{align}
		and 
	\begin{align}
	\|\partial_1\bar H\partial_1\overline{N_2}\|_{L^1(\mathbb{R})}\lesssim & \int_{\mathbb{R}}|\partial_1\bar H|\left[\Big|\partial_1(e^{\tilde E})\Big|\int_{\mathbb{T}^2}(e^H-1-H)dx^{\prime}+e^{\tilde E}\int_{\mathbb{T}^2}|(e^H-1)\partial_1H|dx^{\prime}\right]dx_1
	\notag\\
	\lesssim &\|\partial_1\bar H\|_{L^{\infty}(\mathbb{R})}\left(\| H\|_{L^2(\Omega)}^2+\| \nabla H\|_{L^2(\Omega)}^2\right)
	\lesssim \nu\| H\|_{H^1(\Omega)}^2.
	\end{align}	
	In addition, we have
	\begin{align}
	\|f_2\partial_1^2Z+f_1\partial_1Z\|_{L^1(\mathbb{R})}
	\lesssim \|\partial_1Z\|_{W^{1,\infty}(\mathbb{R})}\|z,r_1\|_{L^2(\Omega)}^2	\lesssim \nu\|z,r_1\|_{L^2(\Omega)}^2
	\end{align}	
	and
	\begin{align}\label{E:4.32}
	\int_{\mathbb{R}} |\overline{L_1}\partial_1Z|dx_1=&\int_{\mathbb{R}}\frac{1}{2}\Big|\left[\int_{\mathbb{T}^2}\left((\partial_1E)^2-(\partial_1\tilde E)^2\right)dx^{\prime}-\int_{\mathbb{T}^2}\left((\partial_2H)^2+(\partial_3 H)^2\right)dx^{\prime}\right]\partial_1Z\Big|dx_1
	\notag\\
	\lesssim& \|\partial_1Z\|_{L^{\infty}(\mathbb{R})}\|\nabla H\|_{L^2(\Omega)}^2+\int_{\mathbb{R}}\int_{\mathbb{T}^2}|\partial_1H|dx^{\prime}|\partial_1\tilde u_1 \partial_1Z|dx_1
	\notag\\
	\lesssim& \|\partial_1Z\|_{L^{\infty}(\mathbb{R})}\|\nabla H\|_{L^2(\Omega)}^2+|\partial_1 \tilde u_1|\left(\|\partial_1 H\|_{L^2(\Omega)}^2+\|\partial_1Z\|_{L^2(\mathbb{R})}\right)
	\notag\\
	\lesssim& (\delta+\nu)\left(\|\nabla H\|_{L^2(\Omega)}^2+\|\partial_1Z\|_{L^2(\mathbb{R})}\right).
	\end{align}	
	Integrating \eqref{E:4.28} over $\mathbb{R}\times [0,t]$ and using \eqref{E:4.31}-\eqref{E:4.32} gives \eqref{E:4.33}. Thus the proof is completed.
	\end{proof}

\begin{lemma}\label{L:4.3}
	For any $T>0$, assume that $(z, {\bf w}, H)(x,t) \in X_{T}$ is the solution to the problem \eqref{E:5.3}-\eqref{E:1.3c}. If $\delta$ and $\nu$ are small, then it holds that
		\begin{align}\label{E:4.40}
	\int_0^t\|\bar H\|_{H^1(\mathbb{R})}^2d\tau\lesssim \int_0^t\|\partial_1Z\|_{L^2(\mathbb{R})}^2d\tau+\nu \int_0^t\|\nabla H\|_{L^2(\Omega)}^2d\tau.
	\end{align}
		\end{lemma}
\begin{proof}
	Multiplying $\eqref{E:4.8}_3$ by $-\bar H$ and integrating over $\mathbb{R}\times [0,t]$ leads to 
	\begin{align*}
\int_0^t\int_{\mathbb{R}}\left[\left(\partial_1\bar H\right)^2+e^{\tilde E}\left(\bar H\right)^2\right]
dx_1d\tau=\int_0^t\int_{\mathbb{R}}\partial_1Z\bar Hdx_1d\tau-\int_0^t\int_{\mathbb{R}}\overline{N_2}\bar Hdx_1d\tau.
	\end{align*}
	It follows that 
	\begin{align*}
\int_0^t\int_{\mathbb{R}}\partial_1Z\bar Hdx_1d\tau\leq \epsilon \int_0^t\|\bar H\|_{L^2(\mathbb{R})}^2d\tau+C_{\epsilon}\int_0^t\|\partial_1Z\|_{L^2(\mathbb{R})}^2d\tau
	\end{align*}
	and 
	\begin{align*}
\int_0^t\int_{\mathbb{R}}\overline{N_2}\bar Hdx_1d\tau\lesssim \|\bar H\|_{L^{\infty}}\int_0^t\|H\|_{L^2(\Omega)}^2d\tau\lesssim \nu \int_0^t\left(\|\bar H\|_{L^2(\mathbb{R})}^2+\|\nabla H\|_{L^2(\Omega)}^2\right)d\tau.
	\end{align*}
By choosing the constant $\epsilon>0$ small enough we get \eqref{E:4.40}. Thus the proof is completed.
	\end{proof}

Now, we take the procedure as $C_2(C_1\eqref{E:4.11}+\eqref{E:4.33})+\eqref{E:4.40}$ with some positive constants $C_1, C_2$ large enough and obtain the following lemma.
\begin{lemma}\label{L:4.4n}
For any $T>0$, assume that $(z, {\bf w}, H)(x,t) \in X_{T}$ is the solution to the problem \eqref{E:5.3}-\eqref{E:1.3c}. If $\delta$ and $\nu$ are small, then it holds that
\begin{align}\label{E:4.41}
		&\|R\|_{L^2(\mathbb{R})}^2+\|Z,\bar H\|_{H^1(\mathbb{R})}^2+\int_0^t \left(\|\sqrt{-\partial_1\tilde u_1}R,\partial_1R,\partial_1Z\|_{L^2(\mathbb{R})}^2+\|\bar H\|_{H^2(\mathbb{R})}^2\right)d\tau
		\notag\\
		\lesssim &\mathbb{N}^2(0)+\nu \|\nabla H\|_{L^2(\Omega)}^2
	 +(\delta+\nu) \int_0^t\left(\|z, r_1, H_t \|_{L^2(\Omega)}^2+\|H\|_{H^1(\Omega)}^2\right) d \tau. 
	\end{align}
\end{lemma}

\section{the estimates of 3-d variables}\label{s3}
This section is devoted to establishing the estimates of the solution $(z, {\bf w}, H)(x, t)$ to the problem \eqref{E:5.3}-\eqref{E:1.3c} under the {\it a priori} assumption \eqref{ass}.
\begin{lemma}\label{L:5.1n}
For any $T>0$, assume that $(z, {\bf w}, H)(x,t) \in X_{T}$ is the solution to the problem \eqref{E:5.3}-\eqref{E:1.3c}. If $\delta$ and $\nu$ are small, then it holds that
	\begin{align}
&\|z, r_1\|_{L^2(\Omega)}^2\lesssim \|\partial_1 Z, \partial_1 R\|_{L^2(\mathbb{R})}^2+\|\nabla z, \nabla r_1\|_{L^2(\Omega)}^2, \label{E:5.1p}\\
		&\|\nabla r_1\|_{L^2(\Omega)} \lesssim(\delta+\nu)\left\|\partial_1 Z, \partial_1 R\right\|_{L^2(\mathbb{R})}+\left\|\nabla w_1\right\|_{L^2(\Omega)}+(\delta+v)\|\nabla z\|_{L^2(\Omega)}, \label{E:5.1}\\
		&\|w_1\|_{L^2(\Omega)} \lesssim\left\|\partial_1 Z, \partial_1 R\right\|_{L^2(\mathbb{R})}+\left\|\nabla w_1\right\|_{L^2(\Omega)}+(\delta+\nu)\|\nabla z\|_{L^2(\Omega)}. \label{E:5.2}
	\end{align}
\end{lemma}
\begin{proof}
It follows from Lemma \ref{l1} that
\begin{align}\label{ebu}
\|f\|_{L^2(\Omega)}^2\leq\|\bar f\|_{L^2(\mathbb{R})}^2+\|f^{\neq}\|_{L^2(\Omega)}^2 \lesssim \|\bar f\|_{L^2(\mathbb{R})}^2+\|\nabla f\|_{L^2(\Omega)}^2.
\end{align}
Thus, \eqref{E:5.1p} follows. 
Since $r_1=m_1-\tilde{m}_1=\rho u_1-\tilde{\rho} \tilde{u}_1=\rho w_1+\tilde{u}_1 z$, we have $$\nabla r_1=\nabla(\tilde{\rho}+z) w_1+\rho \nabla w_1+\nabla \tilde{u}_1 z+\tilde{u}_1 \nabla z.$$ Then using \eqref{ebu} we get
	\begin{align*}
		\left\|\nabla r_1\right\|_{L^2(\Omega)} & \lesssim\left\|\nabla w_1\right\|_{L^2(\Omega)}+(\delta+\nu)\left\|z, w_1\right\|_{L^2(\Omega)}+\delta\|\nabla z\|_{L^2(\Omega)} \notag\\
		& \lesssim\left\|\nabla w_1\right\|_{L^2(\Omega)}+\delta\|\nabla z\|_{L^2(\Omega)}+(\delta+v)\|z, r_1\|_{L^2(\Omega)} \notag\\
		& \lesssim\left\|\nabla w_1\right\|_{L^2(\Omega)}+\delta\|\nabla z\|_{L^2(\Omega)}+(\delta+v)\left(\| \partial_1 Z, \partial_1R\right\|_{L^2(\mathbb{R})}+\left\| \nabla z, \nabla r_1 \|_{L^2(\Omega)}\right),
	\end{align*}
	which leads to \eqref{E:5.1}. Furthermore, it holds that
	\begin{align*}
		\left\|w_1\right\|_{L^2(\Omega)} & \lesssim\left\|r_1\right\|_{L^2(\Omega)}+\delta\|z\|_{L^2(\Omega)} \lesssim\left\|\partial_1 R\right\|_{L^2(\mathbb{R})}+\left\|\nabla r_1\right\|_{L^2(\Omega)}+\delta\left(\left\|\partial_1 Z\right\|_{L^2(\mathbb{R})}+\|\nabla z\|_{\left.L^2(\Omega)\right)}\right) \notag\\
		& \lesssim\left\|\partial_1 R, \partial_1 Z\right\|_{L^2(\mathbb{R})}+\left\|\nabla w_1\right\|_{L^2(\Omega)}+(\delta+\nu)\|\nabla z\|_{L^2(\Omega)}.
	\end{align*}
Thus the proof is completed.
\end{proof}
\subsection{the estimates of electric field}\label{section 3}
\begin{lemma}\label{L:5.2}
For any $T>0$, assume that $(z, {\bf w}, H)(x,t) \in X_{T}$ is the solution to the problem \eqref{E:5.3}-\eqref{E:1.3c}. If $\delta$ and $\nu$ are small, then it holds that
	\begin{align}\label{E:5.22}
		&\|H\|_{H^2(\Omega)}^2+\|H_t \|_{H^1(\Omega)}^2+\int_0^t\|H, H_t\|_{H^2(\Omega)}^2d\tau
	\lesssim \mathbb{N}^2(0)+\|\bar H\|_{L^2(\mathbb{R})}^2+\delta\|z\|_{L^2(\Omega)}^2
	\notag\\
	&+\int_0^t\|\bar H, \bar H_t\|_{L^2(\mathbb{R})}^2d\tau
	 +\left(\delta+\nu\right)\int_0^t\left(\|z\|_{H^1(\Omega)}^2+\|w_1, \nabla {\bf w}\|_{L^2(\Omega)}^2\right)d\tau.
\end{align} 
\end{lemma}
\begin{proof}
	Multiplying $\eqref{E:5.3}_2$ by $\cdot \rho\nabla H$ gives that
	\begin{align}\label{E:5.4}
		&\rho\partial_t{\bf w}\cdot \nabla H+T\nabla z\cdot \nabla H-\Delta {\bf w}\cdot \nabla H+\rho\left(\nabla H\right)^2
		=-({\bf h_1}+{\bf h_2})\cdot \rho\nabla H.
	\end{align}
It follows from $\eqref{E:5.3}_1$ and $\eqref{E:5.3}_3$ that 
\begin{align}\label{E:5.5b}
&\rho\partial_t{\bf w}\cdot \nabla H
\notag\\
=&\left(\dive (\rho{\bf w} H)\right)_t-\left(H\dive (\rho {\bf w})\right)_t-\dive (\rho{\bf w}H_t)+H_t\dive (\rho {\bf w})-\partial_t\rho{\bf w}\cdot \nabla H
\notag\\
=&\left(\dive (\rho{\bf w} H)\right)_t-\dive (\rho{\bf w}H_t)+\left[H(z_t+\dive (z \tilde {\bf u}))\right]_t-H_t(z_t+\dive (z \tilde {\bf u}))-\partial_t\rho{\bf w}\cdot \nabla H
\notag\\
=&\left(\dive (\rho{\bf w} H)\right)_t-\dive (\rho{\bf w}H_t)-\left[H(\Delta H_t-N_{1t}-\dive (z \tilde {\bf u}))\right]_t+H_t(\Delta H_t-N_{1t}-\dive (z \tilde {\bf u}))-\partial_t\rho{\bf w}\cdot \nabla H
\notag\\
=&Q_{1t}+\dive(\cdots)-I_{6}-I_{7},
\end{align}
where $(\cdots)=-H_t(\rho{\bf w}+z\tilde {\bf u})+H_t\nabla H_t-\tilde \rho_t{\bf w}H$, $Q_1=Q_{1,1}+Q_{1,2}$ and
\begin{align}
&Q_{1,1}=\dive\left(H(\rho{\bf w}+z\tilde {\bf u})-H\nabla H_t\right),\notag\\
&Q_{1,2}=\nabla H\cdot \nabla H_t+e^EHH_t+(e^{\tilde E})_tH(e^H-1)-\nabla H z\cdot \tilde {\bf u},\notag\\
&I_{6}=\left(\nabla H_t\right)^2+e^{E}(H_t)^2+\left(e^{\tilde E}\right)_t(e^H-1)H_t-\nabla H_t\cdot z\tilde {\bf u}, \notag\\
&I_{7}=z_t\nabla H\cdot {\bf w}-\partial_1\tilde \rho_tw_1H-\tilde \rho_tH\dive {\bf w}.\nonumber
\end{align}
The computations shows that 
\begin{align*}
	\|I_{6}\|_{L^1(\Omega)}\lesssim \|H_t, \nabla H_t \|_{L^2(\Omega)}^2+(\delta+\nu) \|H, z \|_{L^2(\Omega)}^2
\end{align*}
and
\begin{align*}
\|I_{7}\|_{L^1(\Omega)}
=&\|-\left(\nabla z\cdot {\bf w}+\partial_1 \tilde \rho w_1+\rho \dive {\bf w}+\dive(z \tilde {\bf u})\right)\nabla H\cdot {\bf w}-\partial_1\tilde \rho_tw_1H-\tilde \rho_tH\dive {\bf w}\|_{L^1(\Omega)}
\notag\\
\lesssim &(\delta+\|{\bf w}\|_{L^{\infty}(\Omega)})\|w_1, z, \nabla z, \nabla H, \dive {\bf w}\|_{L^2(\Omega)}^2
\lesssim (\delta+\nu)\|w_1, z, \nabla z, H, \nabla H, \nabla {\bf w}\|_{L^2(\Omega)}^2,
\end{align*}	
which yields that 
\begin{align}\label{E:5.5}
	\int_{\Omega}\rho\partial_t{\bf w}\cdot \nabla Hdx
\gtrsim &\frac{d}{dt}\int_{\Omega}Q_{1,2}dx
	-\|H_t\|_{H^1(\Omega)}^2-(\delta+\nu)\|w_1, z, \nabla z, H, \nabla H, \nabla {\bf w}\|_{L^2(\Omega)}^2.
\end{align}	
From $\eqref{E:5.3}_3$ we have 
\begin{align}\label{O:5.8}
\int_{\Omega}\nabla z\cdot \nabla Hdx
=&\int_{\Omega}\left[(\Delta H)^2+e^{E}(\nabla H)^2+\nabla \left(e^{\tilde E}\right)\cdot (e^H-1)\nabla H\right]dx
\notag\\
\gtrsim&\|\nabla H, \Delta H\|_{L^2(\Omega)}^2-\delta\|H\|_{L^2(\Omega)}^2.
\end{align}
It follows from $\eqref{E:5.3}_1$ and $\eqref{E:5.3}_3$ that 
\begin{align}\label{E:5.11new}
\dive {\bf w}=-\frac{1}{\rho}\left[z_t+\nabla \rho\cdot {\bf w}+\dive (z\tilde {\bf u})\right]=-\frac{1}{\rho}\left[-\Delta H_t+N_{1t}+\nabla \rho\cdot {\bf w}+\dive (z\tilde {\bf u})\right],
\end{align}
which yields that
\begin{align}\label{gw1}
-\Delta {\bf w}\cdot \nabla H=&\left[-\nabla \left(\dive{\bf w}\right)+\nabla\times \mbox{curl} {\bf w}\right] \cdot \nabla H\notag\\
	=&-\dive \left(\dive {\bf w}\nabla H\right)+\dive  {\bf w}\Delta H+\nabla \times \mbox{curl}{\bf w}\cdot \nabla H
	\notag\\
	=&\dive \left(\mbox{curl}{\bf w}\times\nabla H-\dive {\bf w}\nabla H\right)+\dive  {\bf w}\Delta H
	\notag\\
	=&\dive(\cdots)+\left(\frac{(\Delta H)^2}{2\rho}+\frac{e^E}{2\rho}(\nabla H)^2\right)_t+I_{8},
\end{align}
where $(\cdots)=\mbox{curl}{\bf w}\times\nabla H-\dive {\bf w}\nabla H-\frac{e^E}{\rho}H_t\nabla H$ and 
\begin{align*}
I_{8}=&-\left(\frac{1}{2\rho}\right)_t(\Delta H)^2-\left(\frac{e^E}{2\rho}\right)_t(\nabla H)^2+\nabla \left(\frac{e^E}{\rho}\right)H_t\cdot \nabla H-\frac{1}{\rho}\nabla z\cdot{\bf w}\Delta H-\frac{1}{\rho}\partial_1\tilde \rho w_1\Delta H
\notag\\
&-\frac{1}{\rho}\dive (z\tilde {\bf u})\Delta H-\frac{1}{\rho}(e^{\tilde E})_t(e^H-1)\Delta H.
\end{align*}
By the direct calculations we can verify that
\begin{align}
\|I_{8}\|_{L^1(\Omega)}\lesssim& \left(\delta+\|\rho_t, E_t, \nabla \rho, \nabla E, {\bf w}\|_{L^{\infty}(\Omega)}\right)\left(\|w_1, H_t\|_{L^2(\Omega)}^2+\|z\|_{H^1(\Omega)}^2+\|H\|_{H^2(\Omega)}^2\right)
\notag\\
\lesssim & \left(\delta+\|z, {\bf w}\|_{W^{1,\infty}(\Omega)}+\|H_t, \nabla H\|_{L^{\infty}(\Omega)}\right)\left(\|w_1, H_t\|_{L^2(\Omega)}^2+\|z\|_{H^1(\Omega)}^2+\|H\|_{H^2(\Omega)}^2\right)
\notag\\
\lesssim & \left(\delta+\nu\right)\left(\|w_1, H_t\|_{L^2(\Omega)}^2+\|z\|_{H^1(\Omega)}^2+\|H\|_{H^2(\Omega)}^2\right).\nonumber
\end{align}
Thus integrating \eqref{gw1} over $\Omega$ we get 
\begin{align}\label{E:5.7}
		-\int_{\Omega}\Delta {\bf w}\cdot \nabla Hdx\gtrsim&\frac{d}{dt}\int_{\Omega}\left(\frac{1}{2\rho}(\Delta H)^2+\frac{e^{E}}{2\rho}(\nabla H)^2\right)dx
	-\left(\delta+\nu\right)\left(\|w_1, H_t\|_{L^2(\Omega)}^2+\|z\|_{H^1(\Omega)}^2+\|H\|_{H^2(\Omega)}^2\right).
\end{align}
In addition, we have
\begin{align}\label{E:5.15}
&-\int_{\Omega}({\bf h_1}+{\bf h_2})\cdot\rho\nabla Hdx
\notag\\
=&-\int_{\Omega}\left[\tilde {\bf u}\cdot \nabla {\bf w}+{\bf w}\cdot \nabla {\bf w}+\left(\frac{1}{\rho}-\frac{1}{\tilde\rho}\right)\left(T\nabla \tilde \rho-\Delta \tilde {\bf u}\right)+w_1\partial_1\tilde u_1 \right]\cdot \rho\nabla Hdx
\notag\\
\lesssim& \left(\delta+\| {\bf w}\|_{L^{\infty}(\Omega)}\right)\|z, w_1, \nabla {\bf w}, \nabla H\|_{L^2(\Omega)}^2\lesssim \left(\delta+\nu\right)\|z, w_1, \nabla {\bf w}, \nabla H\|_{L^2(\Omega)}^2.
\end{align}
Then  integrating \eqref{E:5.4} over $\Omega\times [0,t]$ and  substituting \eqref{E:5.5}-\eqref{O:5.8} and \eqref{E:5.7}-\eqref{E:5.15} into the results gives that
\begin{align}\label{E:5.11}
&\|\nabla H, \Delta H \|_{L^2(\Omega)}^2
+\int_0^t\|\nabla H, \Delta H\|_{L^2(\Omega)}^2d\tau
\lesssim \mathbb{N}^2(0)+\|H\|_{L^2(\Omega)}^2+\|H_t \|_{H^1(\Omega)}^2+\delta\|z\|_{L^2(\Omega)}^2
\notag\\
&+\int_0^t\|H_t\|_{H^1(\Omega)}^2d\tau+\left(\delta+\nu\right)\int_0^t\left(\|z\|_{H^1(\Omega)}^2+\|w_1, H, \nabla{\bf w}\|_{L^2(\Omega)}^2\right)d\tau.
\end{align}
Recall that it follows from \eqref{ebu} that
\begin{align}
	\|H\|_{L^2(\Omega)}^2\lesssim \|\bar H\|_{L^2(\mathbb{R})}^2+\|\nabla H\|_{L^2(\Omega)}^2,\nonumber
\end{align}
then we use the fact that $\|\Delta H\|_{L^2(\Omega)}\geq c_0\|\nabla^2 H\|_{L^2(\Omega)}$ for some constant $c_0>0$ to transfer \eqref{E:5.11} into
\begin{align}\label{E:5.11b}
&\|H, \nabla^2 H\|_{L^2(\Omega)}^2+\int_0^t\|H\|_{H^2(\Omega)}^2d\tau\lesssim \mathbb{N}^2(0)+\|\bar H\|_{L^2(\mathbb{R})}^2+\|H_t, \nabla H, \nabla H_t\|_{L^2(\Omega)}^2+\delta\|z\|_{L^2(\Omega)}^2
\notag\\
&+\int_0^t\left(\|\bar H\|_{L^2(\mathbb{R})}^2+\|H_t\|_{H^1(\Omega)}^2\right)d\tau +\left(\delta+\nu\right)\int_0^t\left(\|z\|_{H^1(\Omega)}^2+\|w_1, \nabla {\bf w}\|_{L^2(\Omega)}^2\right)d\tau.
\end{align}

Multiplying $\eqref{E:5.3}_2$ by $\cdot \rho\nabla H_t$ gives that
\begin{align}\label{E:5.12}
	&\partial_t\left(\frac{\rho}{2}(\nabla H)^2\right)+\rho\partial_t{\bf w}\cdot \nabla H_t+T\nabla z\cdot \nabla H_t-\Delta {\bf w}\cdot \nabla H_t
	=\left(\frac{\rho}{2}\right)_t(\nabla H)^2-({\bf h_1}+{\bf h_2})\cdot \rho\nabla H_t.
\end{align}
By the similar argument in \eqref{E:5.5} we get
\begin{align}\label{E:5.13}
\rho\partial_t{\bf w}\cdot \nabla H_t=&\dive (H_t(\rho{\bf w})_t)-\partial_t\dive (\rho {\bf w})H_t-\rho_t{\bf w}\cdot \nabla H_t
\notag\\
=&\dive (H_t(\rho{\bf w})_t)+\partial_t(z_t+\dive (z\tilde {\bf u}))H_t-\rho_t{\bf w}\cdot \nabla H_t
\notag\\
=&(\dive(H_t(\rho{\bf w}+z\tilde {\bf u})))_t-\dive (H_{tt}(\rho{\bf w}+z\tilde {\bf u}))+(-\Delta H+N_1)_{tt}H_t-(z\tilde {\bf u})_t\cdot \nabla H_t-\rho_t{\bf w}\cdot \nabla H_t
\notag\\
=&Q_{2t}+\dive(\cdots)-I_{9}-I_{10}-z_t\nabla H_t\cdot {\bf u},
\end{align}	
where $(\cdots)=H_{tt}\nabla H_t-H_{tt}(\rho{\bf w}+z\tilde {\bf u})-\tilde \rho_t{\bf w}H_t$, $Q_2=Q_{2,1}+Q_{2,2}$ and
\begin{align}
&Q_{2,1}=\dive(H_t(\rho{\bf w}+z\tilde {\bf u})-H_t\nabla H_t),\notag\\
&Q_{2,2}=\frac12(\nabla H_t)^2+\frac{e^E}{2}(H_t)^2,\notag\\
&I_{9}=\left(\frac{e^E}{2}\right)_t\left(H_t\right)^2-\left[(e^{\tilde E})_{tt}(e^H-1)+2(e^{\tilde E})_te^HH_t+e^EH_t^2\right]H_t, \notag\\
&I_{10}=z\tilde {\bf u}_t\cdot \nabla H_t-\dive(z\tilde {\bf u}){\bf w}\cdot \nabla H_t-(\partial_1\tilde \rho_tw_1+\tilde \rho_t\dive {\bf w})H_t.\nonumber
\end{align}
By the direct calculation we have
\begin{align}
\|I_{9}\|_{L^1(\Omega)}\lesssim& (\delta+\|H_t\|_{L^{\infty}(\Omega)})\|H, H_t, \nabla H_t\|_{L^2(\Omega)}^2\lesssim (\delta+\nu)\|H, H_t, \nabla H_t\|_{L^2(\Omega)}^2,\label{E:5.20new}
\\
\|I_{10}\|_{L^1(\Omega)}\lesssim& (\delta+\|{\bf w}\|_{L^{\infty}(\Omega)})\left(\|z, H_t\|_{H^1(\Omega)}^2+\|w_1, \nabla {\bf w}\|_{L^2(\Omega)}^2\right)
\notag\\
\lesssim& (\delta+\nu)\left(\|z, H_t\|_{H^1(\Omega)}^2+\|w_1, \nabla {\bf w}\|_{L^2(\Omega)}^2\right)\label{E:5.21new}
\end{align}
and 
\begin{align}\label{E:5.22new}
\|z_t\nabla H_t\cdot {\bf u}\|_{L^1(\Omega)}=&\|(\rho\dive{\bf w}+\nabla z\cdot {\bf w}+\partial_1\tilde \rho w_1+\dive(z\tilde {\bf u}))\nabla H_t\cdot {\bf u}\|_{L^1(\Omega)}
\notag\\
\lesssim& (\delta+\|{\bf w}\|_{L^{\infty}(\Omega)})\|w_1, \nabla {\bf w}, z, \nabla z, \nabla H_t\|_{L^2(\Omega)}^2
 \notag\\
\lesssim& (\delta+\nu)\|w_1, \nabla {\bf w}, z, \nabla z, \nabla H_t\|_{L^2(\Omega)}^2.
\end{align}
Thus we substitute \eqref{E:5.20new}-\eqref{E:5.22new} into \eqref{E:5.13} to get
\begin{align}\label{E:5.23new}
\int_{\Omega}\rho\partial_t{\bf w}\cdot \nabla H_tdx\gtrsim &\frac{d}{dt}\int_{\Omega}Q_{2,2}dx-(\delta+\nu)\left(\|z, H_t\|_{H^1(\Omega)}^2+\|w_1, \nabla {\bf w}, H\|_{L^2(\Omega)}^2\right).
\end{align}

Similar to \eqref{gw1}, we use \eqref{E:5.11new} to get
\begin{align}
-\Delta {\bf w}\cdot \nabla H_t=&\left[-\nabla \left(\dive{\bf w}\right)+\nabla\times \mbox{curl}{\bf w}\right] \cdot \nabla H_t\notag\\
	=&-\dive \left(\dive {\bf w}\nabla H_t\right)+\dive  {\bf w}\Delta H_t+\nabla \times \mbox{curl}{\bf w}\cdot \nabla H_t
	\notag\\
	=&\dive \left(\mbox{curl}{\bf w}\times\nabla H_t-\dive {\bf w}\nabla H_t\right)-\frac{1}{\rho}\left[-\Delta H_t+N_{1t}+\nabla \rho\cdot {\bf w}+\dive (z\tilde {\bf u})\right]\Delta H_t
	\notag\\
	=&\dive(\cdots)+\frac{1}{\rho}(\Delta H_t)^2+\frac{e^E}{\rho}(\nabla H_t)^2-I_{11},\nonumber
\end{align}
where $(\cdots)=\mbox{curl}{\bf w}\times\nabla H_t-\dive {\bf w}\nabla H_t-\frac{e^E}{\rho}H_t\nabla H_t$ and 
\begin{align}
I_{11}=\frac{\Delta H_t}{\rho}\left[(e^{\tilde E})_t(e^H-1)+\nabla z\cdot {\bf w}+\partial_1\tilde \rho w_1+\dive (z\tilde {\bf u})\right]-\nabla \left(\frac{e^E}{\rho}\right)H_t\cdot \nabla H_t. \nonumber
\end{align}
It is easy to check that
\begin{align}
\|I_{11}\|_{L^1(\Omega)}\lesssim& \left(\delta+\|\nabla z, \nabla H, {\bf w}\|_{L^{\infty}(\Omega)}\right)\left(\|w_1, H\|_{L^2(\Omega)}^2+\|z\|_{H^1(\Omega)}^2+\|H_t\|_{H^2(\Omega)}^2\right)
\notag\\
\lesssim& \left(\delta+\nu\right)\left(\|w_1, H\|_{L^2(\Omega)}^2+\|z\|_{H^1(\Omega)}^2+\|H_t\|_{H^2(\Omega)}^2\right), \nonumber
\end{align}
together with $\|\Delta H_t\|_{L^2(\Omega)}\geq c_0\|\nabla^2 H_t\|_{L^2(\Omega)}$, which leads to 
\begin{align}\label{E:5.14}
-\int_{\Omega}\Delta {\bf w}\cdot \nabla H_tdx	\gtrsim&\|\nabla H_t\|_{H^1(\Omega)}^2
-\left(\delta+\nu\right)\left(\|w_1, H\|_{L^2(\Omega)}^2+\|z\|_{H^1(\Omega)}^2+\|H_t\|_{H^2(\Omega)}^2\right).
\end{align}
In addition, it follows from  $\eqref{E:5.3}_3$ that
\begin{align}
\nabla z\cdot \nabla H_t=-\dive (\Delta H \nabla H_t)+\frac{1}{2}\left(e^E(\nabla H)^2+(\Delta H)^2\right)_t-\left(\frac{e^E}{2}\right)_t(\nabla H)^2+\nabla (e^{\tilde E})(e^H-1)\cdot \nabla H_t,\nonumber
\end{align}
which leads to 
\begin{align}
	\int_{\Omega}T\nabla z\cdot \nabla H_tdx
	\gtrsim& \frac{T}{2}\frac{d}{dt}\int_{\Omega}\left[e^E(\nabla H)^2+(\Delta H)^2\right]dx-(\delta+\|H_t\|_{L^{\infty}(\Omega)})\|H, \nabla H, \nabla H_t\|_{L^2(\Omega)}^2
		\notag\\
	\gtrsim& \frac{T}{2}\frac{d}{dt}\int_{\Omega}\left[e^E(\nabla H)^2+(\Delta H)^2\right]dx-(\delta+\nu)\|H, \nabla H, \nabla H_t\|_{L^2(\Omega)}^2.
	\end{align}
	And by the straightforward computations we get
	\begin{align}\label{E:5.18}
		-\int_{\Omega}({\bf h_1}+{\bf h_2})\cdot \rho\nabla H_tdx
	\lesssim \left(\delta+\|{\bf w}\|_{L^{\infty}(\Omega)}\right)\| z, w_1, \nabla {\bf w}, \nabla H_t\|_{L^2(\Omega)}^2\lesssim \left(\delta+\nu\right)\| z, w_1, \nabla {\bf w}, \nabla H_t\|_{L^2(\Omega)}^2.
\end{align}
Thus,  integrating \eqref{E:5.12}  over $\Omega\times [0,t]$ and substituting \eqref{E:5.23new}-\eqref{E:5.18} into the resultants gives that
\begin{align}\label{E:5.28new}
	&\|H_t, \nabla H\|_{H^1(\Omega)}^2
	+\int_0^t\|\nabla H_t\|_{H^1(\Omega)}^2d\tau
	\lesssim \mathbb{N}^2(0)+\left(\delta+\nu \right)\int_0^t\left(\|w_1, \nabla {\bf w}, H_t\|_{L^2(\Omega)}^2+\|z, H\|_{H^1(\Omega)}^2\right)d\tau,
\end{align}
where we have used the fact that $\|\Delta H\|_{L^2(\Omega)}\geq c_0\|\nabla^2 H\|_{L^2(\Omega)}$ for some positive constant $c_0$.
Recall that it follows from \eqref{ebu} that \[\|H_t\|_{L^2(\Omega)}^2\lesssim \|\bar H_t\|_{L^2(\mathbb{R})}^2+\|\nabla H_t\|_{L^2(\Omega)}^2,\]
then \eqref{E:5.28new} can be converted into 
\begin{align}\label{E:5.19}
	&\|H_t, \nabla H\|_{H^1(\Omega)}^2
	+\int_0^t\|H_t\|_{H^2(\Omega)}^2d\tau
	\notag\\
	\lesssim& \mathbb{N}^2(0)+\int_0^t\|\bar H_t\|_{L^2(\mathbb{R})}^2d\tau+\left(\delta+\nu \right)\int_0^t\left(\|w_1, \nabla {\bf w}\|_{L^2(\Omega)}^2+\|z, H\|_{H^1(\Omega)}^2\right)d\tau.
\end{align}
Then taking the procedure as $C_3\eqref{E:5.19}+\eqref{E:5.11b}$ with some positive constant $C_3$ large enough leads to \eqref{E:5.22}.
Thus, the proof is completed.
\end{proof}

\begin{lemma}\label{L:5.2}
For any $T>0$, assume that $(z, {\bf w}, H)(x,t) \in X_{T}$ is the solution to the problem \eqref{E:5.3}-\eqref{E:1.3c}. If $\delta$ and $\nu$ are small, then it holds that
	\begin{align}\label{E:4.37}
	\|\bar H_t\|_{H^1(\mathbb{R})}^2\lesssim \|\partial_1 R\|_{L^2(\mathbb{R})}^2+\delta \|\bar H\|_{L^2(\mathbb{R})}^2+\nu\|\nabla H_t\|_{L^2(\Omega)}^2.
	\end{align}
\end{lemma}
\begin{proof}
Differentiating $\eqref{E:2.2}_3$ with respect to $t$, and multiplying the resultant by $\bar H_t$ then integrating the resultant over $\mathbb{R}$ gives that
	\begin{align*}
	\int_{\mathbb{R}}(\partial_1 \bar H_t)^2dx_1+\int_{\mathbb{R}}e^{\tilde E}(\bar H_t)^2dx_1= \int_{\mathbb{R}}\partial_1 R\partial_1 \bar H_tdx_1-\int_{\mathbb{R}}\left(e^{\tilde E}\right)_t\bar H\bar H_tdx_1-\int_{\mathbb{R}}N_{1t}^{od}\bar H_tdx_1,
	\end{align*}
	which yields that
	\begin{align*}
		\| \bar H_t\|_{H^1(\mathbb R)}\lesssim& \|\partial_1 R\|_{L^2(\mathbb{R})}^2+\delta \|\bar H\|_{L^2(\mathbb{R})}^2+\|H\|_{L^{\infty}(\Omega)}\left(\|H_t\|_{L^2(\Omega)}^2+\|\bar H_t\|_{L^2(\mathbb{R})}^2\right)
		\notag\\
		\lesssim&\|\partial_1 R\|_{L^2(\mathbb{R})}^2+\delta \|\bar H\|_{L^2(\mathbb{R})}^2+\nu\left(\|\nabla H_t\|_{L^2(\Omega)}^2+\|\bar H_t\|_{L^2(\mathbb{R})}^2\right).
	\end{align*}
Then \eqref{E:4.37} follows due to the smallness of $\nu$.
 Thus, the proof is completed.
	\end{proof}
\subsection{the estimates of $z, {\bf w}$}
\begin{lemma}
	For any $T>0$, assume that $(z, {\bf w}, H)(x,t) \in X_{T}$ is the solution to the problem \eqref{E:5.3}-\eqref{E:1.3c}. If $\delta$ and $\nu$ are small, then it holds that
	\begin{align}\label{E:5.30}
		\|z, {\bf w},  H, \nabla H\|_{L^2(\Omega)}^2+\int_0^t\|\nabla {\bf w}\|_{L^2(\Omega)}^2d\tau\lesssim \mathbb{N}^2(0)+(\delta+\nu)\int_0^t\left(\|z, H\|_{H^1(\Omega)}^2+\|w_1\|_{L^2(\Omega)}^2\right)d\tau.
	\end{align}
\end{lemma}
\begin{proof}
	Multiplying  $\eqref{E:5.3}_2$ by $\cdot \rho{\bf w}$ yields that
	\begin{align} \label{E:5.26}
		& \frac12\left(\rho |{\bf w}|^2+e^{E} H^2+(\nabla H)^2\right)_t+|\nabla {\bf w}|^2+T \nabla z \cdot {\bf w}  = \operatorname{div}(\cdots)+I_{12}-{\bf h_2}\cdot \rho{\bf w},
	\end{align}
where $(\cdots)=\frac{1}{2}\nabla \left(|{\bf w}|^2\right)-H(\rho {\bf w}+z \tilde {\bf u})+H\nabla H_t-\frac{\rho{\bf u}}{2}|{\bf w}|^2$
and 
\begin{align}
I_{12}=\nabla H \cdot z \tilde{{\bf u}}+\left(\frac{e^E}{2}\right)_tH^2-(e^{\tilde E})_t(e^H-1)H.\nonumber
\end{align}
To deal with the term $T \nabla z \cdot {\bf w}$, we multiply $\eqref{E:5.3}_1$ by $\frac{Tz}{\rho}$ to get
\begin{align}\label{E:5.27}
	\left(\frac{T}{2 \rho} z^2\right)_t+T z \dive{\bf w}=\left(\frac{T}{2 \rho}\right)_tz^2-[\nabla z\cdot {\bf w}+\partial_1\tilde \rho w_1+\dive(z\tilde {\bf u})]\frac{Tz}{\rho}=:\left(\frac{T}{2 \rho}\right)_tz^2+I_{13}.
\end{align}
Adding \eqref{E:5.26} and \eqref{E:5.27} up gives that 
	\begin{align} \label{E:5.28}
	& \frac12\left(\rho |{\bf w}|^2+e^{{E}} H^2+(\nabla H)^2+\frac{T}{\rho} z^2\right)_t+|\nabla {\bf w}|^2  = \operatorname{div}(\cdots)+I_{12}+I_{13}-{\bf h_2}\cdot \rho{\bf w}.
\end{align}
It follows from 
\[{\bf h_2}\cdot \rho{\bf w}=\rho \partial_1 \tilde u_1w_1^2+\frac{\partial_1^2 \tilde u_1zw_1}{\tilde \rho}-\frac{T\partial_1\tilde \rho zw_1}{\tilde \rho}\]
that 
\begin{align}\label{E:5.29}
\|I_{12}+I_{13}-{\bf h_2}\cdot \rho{\bf w}\|_{L^1(\Omega)}\lesssim& (\delta+\|{\bf w}, H_t\|_{L^{\infty}(\Omega)})\left(\|z, H\|_{H^1(\Omega)}^2+\|w_1\|_{L^2(\Omega)}^2\right)
\notag\\
\lesssim&(\delta+\nu)\left(\|z, H\|_{H^1(\Omega)}^2+\|w_1\|_{L^2(\Omega)}^2\right).
\end{align}
Thus integrating \eqref{E:5.28} over $\Omega\times [0,t]$ and using \eqref{E:5.29} we conclude \eqref{E:5.30}. Therefore, the proof is completed.
	\end{proof}
\begin{lemma}\label{L:5.4}
For any $T>0$, assume that $(z, {\bf w}, H)(x,t) \in X_{T}$ is the solution to the problem \eqref{E:5.3}-\eqref{E:1.3c}. If $\delta$ and $\nu$ are small, then it holds that
				\begin{align}\label{E:5.40}
		&\|\nabla z, \nabla {\bf w}\|_{L^2(\Omega)}^2+\int_0^t\|\nabla z, \nabla^2 {\bf w}\|_{L^2(\Omega)}^2d\tau
		\notag\\
		\lesssim& \mathbb{N}^2(0)+\| {\bf w}\|_{L^{2}(\Omega)}^2+\int_0^t	\| \nabla {\bf w}, \nabla H\|_{L^{2}(\Omega)}^2 d\tau+(\delta+\nu)\int_0^t\|z, w_1\|_{L^2(\Omega)}^2d\tau.
	\end{align}
\end{lemma}
\begin{proof}
We multiply $\partial_x^{\alpha}\eqref{E:5.3}_2$ by $\cdot\partial_x^{\alpha}{\bf w}$ with the multi-index $|\alpha|=1$ to get
	\begin{align}\label{E:5.31}
	& \partial_t\left(\frac{1}{2}|\partial_x^\alpha {\bf w}|^2\right)-\frac{T }{\rho}\partial_x^\alpha z\partial_x^\alpha(\dive {\bf w})+ \frac{|\nabla \partial_x^\alpha{\bf w}|^2}{\rho}
	=\dive (\cdots)+\partial_x^{\alpha}H\dive \partial_x^{\alpha}{\bf w}-\partial_x^\alpha({\bf h}_1+{\bf h}_2)\cdot\partial_x^{\alpha}{\bf w}+I_{14},
\end{align}
where $(\cdots)=-\frac{T\partial_x^{\alpha}z\partial_x^{\alpha}{\bf w}}{\rho}+\frac{\nabla\left(\partial_x^{\alpha}{\bf w}\right)\cdot \partial_x^{\alpha}{\bf w}}{\rho}-\partial_x^{\alpha}H\partial_x^{\alpha}{\bf w}$
and 
\begin{align*}
	I_{14}=&\nabla \left(\frac{T}{\rho}\right)\partial_x^\alpha z \cdot \partial_x^\alpha{\bf w} -\nabla \partial_x^\alpha{\bf w}\cdot  \partial_x^\alpha{\bf w}\cdot \nabla \left(\frac{1}{\rho}\right)	
	+ \sum_{|\beta|=1, \beta \leq \alpha} \left[\partial_x^{\beta}\left(\frac{1}{\rho}\right)\Delta {\bf w}-\partial_x^{\beta}\left(\frac{T}{\rho}\right)\nabla z\right]\cdot \partial_x^\alpha{\bf w}.
\end{align*}
It follows from $\eqref{E:5.3}_1$ that 
\begin{align}\label{E:5.33}
	\partial_x^\alpha(\operatorname{div} {\bf w})=\frac{-1}{\rho}\left\{\partial_x^\alpha z_t+\sum_{|\beta|=1, \beta \leq \alpha} \partial_x^\beta \rho \partial_x^{\alpha-\beta}(\operatorname{div} {\bf w})+\partial_x^\alpha(\nabla z \cdot {\bf u}+\partial_1\tilde \rho w_1+z \dive \tilde{\bf u})\right\} .
\end{align}
By the calculations we get
\begin{align*}
		& -\frac{T}{\rho} \partial_x^\alpha z\left(-\frac{1}{\rho}\right) \partial_x^\alpha z_t=\frac{T}{\rho^2}\left(\frac{1}{2}\left(\partial_x^\alpha z\right)^2\right)_t=\left(\frac{T}{2 \rho^2}\left(\partial_x^\alpha z\right)^2\right)_t-\left(\frac{T}{\rho^2}\right)_t \frac{1}{2}\left(\partial_x^\alpha z\right)^2
\end{align*}
and 
\begin{align*}
		& -\frac{T}{\rho} \partial_x^\alpha z\left(-\frac{1}{\rho}\right) \partial_ x^\alpha(\nabla z \cdot {\bf u})
	= \dive\left(\frac{T {\bf u}\left(\partial_x^\alpha z\right)^2}{2 \rho^2}\right)-\dive \left(\frac{T {\bf u}}{2 \rho^2}\right)\left(\partial_x^\alpha z\right)^2+\sum_{|\beta|=1, \beta \leq \alpha}\frac{T}{\rho^2} \partial_x^\alpha z\partial_x^{\beta} {\bf u}\cdot \partial_x^{\alpha-\beta}\nabla z,
\end{align*}
which yields that 
\begin{align}\label{E:5.32}
-\int_{\Omega}\frac{T }{\rho}\partial_x^\alpha z\partial_x^\alpha(\dive {\bf w})dx\gtrsim & \frac{d}{dt}\int_{\Omega}\frac{T}{2 \rho^2}\left(\partial_x^\alpha z\right)^2dx-\left(\delta+\|{\bf w}, \nabla z, \nabla {\bf w}\|_{L^{\infty}(\Omega)}\right)\|z, w_1, \nabla z, \nabla {\bf w}\|_{L^{2}(\Omega)}^2
\notag\\
\gtrsim & \frac{d}{dt}\int_{\Omega}\frac{T}{2 \rho^2}\left(\partial_x^\alpha z\right)^2dx-\left(\delta+\nu\right)\|z, w_1, \nabla z, \nabla {\bf w}\|_{L^{2}(\Omega)}^2.
\end{align}
In addition, we can verify that 
\begin{align}
	-\int_{\Omega}\partial_x^\alpha({\bf h}_1+{\bf h}_2)\partial_x^{\alpha}{\bf w}dx
	\lesssim&(\delta+\nu)\|z, w_1, \nabla z, \nabla {\bf w}, \nabla^2 {\bf w}\|_{L^{2}(\Omega)}^2, \\
	\int_{\Omega}\partial_x^{\alpha}H\dive \partial_x^{\alpha}{\bf w}dx\leq &\epsilon\|\nabla \partial_x^\alpha {\bf w}\|_{L^{2}(\Omega)}^2 +C_{\epsilon}\| \nabla H\|_{L^{2}(\Omega)}^2 
\end{align}
and 
\begin{align}\label{E:5.34}
	\|I_{14}\|_{L^1(\Omega)}\lesssim& \left(\delta+\|\nabla {\bf w}, \nabla z\|_{L^{\infty}(\Omega)}\right)\| \nabla z, \nabla {\bf w}, \nabla^2 {\bf w}\|_{L^{2}(\Omega)}^2
	\lesssim \left(\delta+\nu\right)\| \nabla z, \nabla {\bf w}, \nabla^2 {\bf w}\|_{L^{2}(\Omega)}^2.
\end{align}
Then integrating \eqref{E:5.31} over $\Omega\times [0,t]$ and using \eqref{E:5.32}-\eqref{E:5.34}, we choose some constant $\epsilon>0$ small enough to get
\begin{align}\label{E:5.36}
	&\|\nabla z, \nabla {\bf w}\|_{L^{2}(\Omega)}^2+\int_0^t	\|\nabla^2 {\bf w}\|_{L^{2}(\Omega)}^2d\tau
	\notag\\
	\lesssim& \mathbb{N}^2(0)+C_{\epsilon}\int_0^t	\|\nabla H\|_{L^{2}(\Omega)}^2 d\tau
	+\left(\delta+\nu\right)\int_0^t\|z, w_1, \nabla z, \nabla {\bf w}\|_{L^{2}(\Omega)}^2d\tau.
\end{align}

	Multiplying $\eqref{E:5.3}_2$ by $\cdot \nabla z$ gives that 
	\begin{align}\label{E:5.38}
	\frac{T}{\rho}(\nabla z)^2=&-	({\bf w}\cdot \nabla z)_t+\dive({\bf w}z_t)+\dive {\bf w}\left[{\bf w}\cdot \nabla z+\partial_1\tilde \rho w_1+\rho \dive {\bf w}+\dive (z\tilde {\bf u})\right]
	\notag\\
	&-({\bf h}_1+{\bf h}_2)\cdot \nabla z-\nabla H\cdot \nabla z+\frac{\Delta {\bf w}}{\rho}\cdot \nabla z.
	\end{align}
Immediately, we integrate \eqref{E:5.38} over $
\Omega\times [0, t]$ and use $\|{\bf w}\|_{L^{\infty}(\Omega)}\lesssim \nu$ to get
\begin{align}\label{E:5.39}
	&\int_0^t\|\nabla z\|_{L^2(\Omega)}^2d\tau
	\lesssim \mathbb{N}^2(0)+\|{\bf w}, \nabla z\|_{L^2(\Omega)}^2+\int_0^t\|\nabla {\bf w}, \nabla^2 {\bf w}, \nabla H\|_{L^2(\Omega)}^2d\tau+\delta\int_0^t\|z, w_1\|_{L^2(\Omega)}^2d\tau.
\end{align}
Taking the procedure as $C_4\eqref{E:5.36}+\eqref{E:5.39}$ with some positive constnat $C_4$ large enough leads to \eqref{E:5.40} directly. Thus, the proof is completed.
\end{proof}

Based on Lemma \ref{L:4.4n} and Lemmas \ref{L:5.1n}-\ref{L:5.4} we derive the estimates for the anti-derivatives and the lower-order derivatives of the solution $(z, {\bf w}, H)(x,t)$.
\begin{lemma}\label{L:5.6}
	For any $T>0$, assume that $(z, {\bf w}, H)(x,t) \in X_{T}$ is the solution to the problem \eqref{E:5.3}-\eqref{E:1.3c}. If $\delta$ and $\nu$ are small, then it holds that
		\begin{align}\label{E:5.41}
		&\|R\|_{L^2(\mathbb{R})}^2+\|Z,\bar H\|_{H^1(\mathbb{R})}^2+\|H\|_{H^2(\Omega)}^2+\|z, {\bf w}, H_t\|_{H^1(\Omega)}^2+\int_0^t\left(\|H, H_t\|_{H^2(\Omega)}^2+\|\nabla z, \nabla {\bf w}, \nabla^2 {\bf w}\|_{L^2(\Omega)}^2\right)d\tau
	\notag\\
	&+\int_0^t\left(\|\sqrt{-\partial_1\tilde u_1}R,\partial_1R,\partial_1Z\|_{L^2(\mathbb{R})}^2+\|\bar H_t\|_{H^1(\mathbb{R})}^2+\|\bar H\|_{H^2(\mathbb{R})}^2\right)d\tau
		\lesssim \mathbb{N}^2(0).
	\end{align}
\end{lemma}
\begin{proof}
We take the procedure as $C_5\int_0^t\eqref{E:4.37}d\tau+\eqref{E:5.22}+C_6\eqref{E:4.41}$ with some positive constants $C_5, C_6$ large enough and obtain 
	\begin{align}\label{E:5.40p}
&\|R\|_{L^2(\mathbb{R})}^2+\|Z,\bar H\|_{H^1(\mathbb{R})}^2+\|H\|_{H^2(\Omega)}^2+\|H_t \|_{H^1(\Omega)}^2
\notag\\
&+\int_0^t \left(\|\sqrt{-\partial_1\tilde u_1}R,\partial_1R,\partial_1Z\|_{L^2(\mathbb{R})}^2+\|\bar H_t\|_{H^1(\mathbb{R})}^2+\|\bar H\|_{H^2(\mathbb{R})}^2+\|H, H_t\|_{H^2(\Omega)}^2\right)d\tau
\notag\\
\lesssim&\mathbb{N}^2(0)+\delta\|z\|_{L^2(\Omega)}^2
+(\delta+\nu) \int_0^t\left(\|z, \nabla {\bf w}\|_{H^1(\Omega)}^2+\|r_1, w_1\|_{L^2(\Omega)}^2\right) d \tau.
\end{align}
	Then taking the procedure as $C_7\eqref{E:5.30}+\eqref{E:5.40}+C_8\eqref{E:5.40p}$ with some positive constants $C_7, C_8$ large enough, we have
		\begin{align}\label{E:5.41bo}
		&\|R\|_{L^2(\mathbb{R})}^2+\|Z,\bar H\|_{H^1(\mathbb{R})}^2+\|H\|_{H^2(\Omega)}^2+\|z, {\bf w}, H_t\|_{H^1(\Omega)}^2
	\notag\\
	&+\int_0^t\left(\|\sqrt{-\partial_1\tilde u_1}R,\partial_1R,\partial_1Z\|_{L^2(\mathbb{R})}^2+\|\bar H_t\|_{H^1(\mathbb{R})}^2+\|\bar H\|_{H^2(\mathbb{R})}^2\right)d\tau
	\notag\\
	&+\int_0^t\left(\|H, H_t\|_{H^2(\Omega)}^2
	+\|\nabla z, \nabla {\bf w}, \nabla^2 {\bf w}\|_{L^2(\Omega)}^2\right)d\tau
		\lesssim \mathbb{N}^2(0)+(\delta+\nu)\int_0^t\|z, r_1, w_1\|_{L^2(\Omega)}^2d\tau.
	\end{align}
Finally, we substitute \eqref{E:5.1p}-\eqref{E:5.2} into \eqref{E:5.41bo}
to obtain \eqref{E:5.41}.
Thus, the proof is completed.
\end{proof}

Furthermore, we can obtain the estimates for the higher-order derivatives of $(z, {\bf w}, H)(x,t)$ by the similar way shown in Lemmas \ref{L:5.2} and \ref{L:5.4}.
\begin{lemma}\label{L:5.7}
	For any $T>0$, assume that $(z, {\bf w}, H)(x,t) \in X_{T}$ is the solution to the problem \eqref{E:5.3}-\eqref{E:1.3c}. If $\delta$ and $\nu$ are small, then it holds that
	\begin{align}
		&\|\nabla^{|\alpha|} z, \nabla^{|\alpha|} {\bf w}\|_{L^{2}(\Omega)}^2+\int_0^t	\|\nabla^{|\alpha|+1} {\bf w}, \nabla^{|\alpha|} z\|_{L^{2}(\Omega)}^2d\tau
		\lesssim \mathbb{N}^2(0),\quad |\alpha|=2, 3, 4,\label{E:5.42}\\
			&\|\nabla^{|\beta|} H\|_{H^2(\Omega)}^2+\|\nabla^{|\beta|} H_t \|_{H^1(\Omega)}^2+\int_0^t\left[\|\nabla^{|\beta|} H\|_{H^2(\Omega)}^2+\|\nabla^{|\beta|+1} H_t\|_{H^1(\Omega)}^2\right]d\tau
		\lesssim \mathbb{N}^2(0), \quad |\beta|=1, 2, 3.\label{E:5.42a}
	\end{align}
\end{lemma}
\begin{proof}
 We take the procedure as $\partial_x^{\alpha}\eqref{E:5.3}_2\cdot \partial_x^{\alpha}{\bf w}$, $\partial_x^{\beta}\eqref{E:5.3}_2$ by $\cdot \partial_x^{\beta}\nabla z$, as well as $\partial_x^{\beta}\left(\rho\eqref{E:5.3}_2\right)\cdot (\partial_x^{\beta}\nabla H+\bar C_{\beta} \cdot \partial_x^{\beta}\nabla H_t)$ for the positive constants $\bar C_{\beta}$ large suitable and then apply Lemma \ref{L:5.6} to hence that \eqref{E:5.42}-\eqref{E:5.42a} holds with the multi-indexs $|\alpha|=2$ and $|\beta|=1$. By analogy with this situation, we can use the similar arguments to prove that \eqref{E:5.42}-\eqref{E:5.42a} holds for the case of multi-indexs $|\alpha|=3$ and $|\beta|=2$ as well as $|\alpha|=4$ and $|\beta|=3$ respectively. Thus, the proof is completed.
\end{proof}

Finally, let $\nu=c_1\mathbb{N}(0)$ with some positive constant $c_1$ large enough. When the initial data satisfies \eqref{zxx} and $\nu_0$ is small, namely, $\mathbb{N}(0)\lesssim \nu_0\ll 1$, it follows from Lemmas \ref{L:5.6} and \ref{L:5.7} that
\begin{align*}
\mathbb{N}(T)=\sup_{t \in(0, T)}\left(\left\|Z, R\right\|_{L^2(\mathbb{R})}+\|z, {\bf w}, H_t\|_{H^4(\Omega)}+\|H\|_{H^5(\Omega)}\right)\leq \frac{\nu}{2}<\nu,
\end{align*}
which closes the {\it a priori} assumption \eqref{ass}. Thus, the global existence of the
solution $(z, {\bf w}, H)(x,t) \in X_{\infty}$ has been obtained.
Moreover, it holds that 
\begin{align}\label{E:5.43g}
\|z, {\bf w}, H, H_t\|_{H^4(\Omega)}\leq \frac{\nu}{2}<\nu, \quad \forall\, t>0.
\end{align}
In addition, it follows from $\eqref{E:5.3}_3$ that $H \in C(0, +\infty; H^6(\Omega))$ and $H_t \in C(0, +\infty; H^5(\Omega))$ with \[\|H\|_{H^6(\Omega)}+\|H_t\|_{H^5(\Omega)}\lesssim \nu.\]

\section{the time-asymptotic stability}\label{s4}
Now, to complete the proof of Theorem \ref{theorem1.1}, we still need to add proofs for \eqref{E:1.15} and \eqref{E:1.17}. From Lemma \ref{de} we decompose $(z, {\bf w}, H, H_t)$ as 
$$
(z, {\bf w}, H, H_t)=\sum_{k=1}^3\left(z^{(k)}, {\bf w}^{(k)}, H^{(k)}, H_t^{(k)}\right),
$$
where each $\left(z^{(k)}, {\bf w}^{(k)}, H^{(k)}, H_t^{(k)}\right)$ satisfies the $k$-dimensional G-N inequalities. Then, it holds that
\begin{align}
	\left\|z^{(1)}, {\bf w}^{(1)}, H^{(1)}, H_t^{(1)}\right\|_{L^{\infty}(\Omega)} & \lesssim\left\|\nabla\left(z^{(1)}, {\bf w}^{(1)}, H^{(1)}, H_t^{(1)}\right)\right\|_{L^2(\Omega)}^{\frac{1}{2}}\left\|z^{(1)}, {\bf w}^{(1)}, H^{(1)}, H_t^{(1)}\right\|_{L^2(\Omega)}^{\frac{1}{2}} \notag\\
	& \lesssim\|\nabla(z, {\bf w}, H, H_t)\|_{L^2(\Omega)}^{\frac{1}{2}}\|z, {\bf w}, H, H_t\|_{L^2(\Omega)}^{\frac{1}{2}}, \label{E:6.1a}\\
	\left\|z^{(2)}, {\bf w}^{(2)}, H^{(2)}, H_t^{(2)}\right\|_{L^{\infty}(\Omega)} & \lesssim\left\|\nabla\left(z^{(2)}, {\bf w}^{(2)}, H^{(2)}, H_t^{(2)}\right)\right\|_{L^2(\Omega)} \lesssim\|\nabla(z, {\bf w}, H, H_t)\|_{L^2(\Omega)},\label{E:6.2c}
\end{align}
and
\begin{align}\label{E:6.3a}
	\left\|z^{(3)}, {\bf w}^{(3)}, H^{(3)}, H_t^{(3)}\right\|_{L^{\infty}(\Omega)} & \lesssim\left\|\nabla^2\left(z^{(3)}, {\bf w}^{(3)}, H^{(3)}, H_t^{(3)}\right)\right\|_{L^2(\Omega)}^{\frac{1}{2}}\left\|z^{(3)}, {\bf w}^{(3)}, H^{(3)}, H_t^{(3)}\right\|_{L^6(\Omega)}^{\frac{1}{2}} 
	\notag\\
	&\lesssim\left\|\nabla^2\left(z^{(3)}, {\bf w}^{(3)}, H^{(3)}, H_t^{(3)}\right)\right\|_{L^2(\Omega)}^{\frac{1}{2}}\left\|\nabla\left(z^{(3)}, {\bf w}^{(3)}, H^{(3)}, H_t^{(3)}\right)\right\|_{L^2(\Omega)}^{\frac{1}{2}} 
\notag\\
& \lesssim\left\|\nabla^2(z, {\bf w}, H, H_t)\right\|_{L^2(\Omega)}^{\frac{1}{2}}\|\nabla(z, {\bf w}, H, H_t)\|_{L^2(\Omega)}^{\frac{1}{2}}.
\end{align}
Then from \eqref{E:6.1a}-\eqref{E:6.3a}, one has that
\begin{align}\label{E:6.4a}
	\|(z, {\bf w}, H, H_t)(t)\|_{L^{\infty}\left(\mathbb{R}^3\right)}=\|(z, {\bf w}, H, H_t)(t)\|_{L^{\infty}(\Omega)} \lesssim\|\nabla(z, {\bf w}, H, H_t)(t)\|_{L^2(\Omega)}^{\frac{1}{2}}.
\end{align}
Since $\|\nabla(z, {\bf w}, H, H_t)(t)\|_{L^2(\Omega)}^2 \in W^{1,1}((0,+\infty))$, one can get that
$$
\|\nabla(z, {\bf w}, H, H_t)(t)\|_{L^2(\Omega)}^2 \rightarrow 0 \quad \text { as } t \rightarrow+\infty,
$$
which yields that 
\[\|(z, {\bf w}, H, H_t)(t)\|_{L^{\infty}\left(\mathbb{R}^3\right)}  \rightarrow 0 \quad \mbox{as} \quad t \rightarrow+\infty.\]
Similarly, since $\left\|\nabla^2(z, {\bf w}, H, H_t)(t)\right\|_{L^2(\Omega)}^2$ and $\left\|\nabla^3(z, {\bf w}, H, H_t)(t)\right\|_{L^2(\Omega)}^2$ belong to the $W^{1,1}((0,+\infty))$ space respectively, one can also get that
\begin{align*}
\|\nabla(z, {\bf w}, H, H_t)(t)\|_{L^{\infty}\left(\mathbb{R}^3\right)} \lesssim\left\|\nabla^2(z, {\bf w}, H, H_t)(t)\right\|_{L^2(\Omega)}^{\frac12} \rightarrow 0 \quad \mbox{as} \quad t \rightarrow+\infty,\notag\\
\|\nabla^2(z, {\bf w}, H, H_t)(t)\|_{L^{\infty}\left(\mathbb{R}^3\right)} \lesssim\left\|\nabla^3(z, {\bf w}, H, H_t)(t)\right\|_{L^2(\Omega)}^{\frac12} \rightarrow 0 \quad \mbox{as} \quad t \rightarrow+\infty.
\end{align*} 
Furthermore, it follows from $\eqref{E:5.3}_3$ and the fact $\|\Delta H, \Delta H_t\|_{L^{2}\left(\Omega\right)}\gtrsim c_0 \|\nabla^2H, \nabla^2H_t\|_{L^{2}\left(\Omega\right)}$ that
\[\|\nabla^2 H\|_{H^{4}\left(\Omega\right)}+\|\nabla^2 H_t\|_{H^{3}\left(\Omega\right)}\lesssim \|z, {\bf w}, H\|_{H^{4}\left(\Omega\right)}+\|H_t\|_{H^{3}\left(\Omega\right)},\]
which implies that 
\begin{align*}
\|\nabla^2H(t)\|_{W^{2,\infty}\left(\mathbb{R}^3\right)}+\|\nabla^2H_t(t)\|_{W^{1,\infty}\left(\mathbb{R}^3\right)} 
\rightarrow 0 \quad \mbox{as} \quad t \rightarrow+\infty.
\end{align*} 
Note that ${\bf m}-\tilde{{\bf m}}=\rho{\bf w}+z \tilde{{\bf u}}$, thus we obtain \eqref{E:1.15}. 

Then, it remains to show the exponential decay rate of $\left(z^{\neq}, {\bf w}^{\neq}, H^{\neq}, H^{\neq}_t\right)$.  It follows from Sobolev inequality,  Lemma \ref{l1} and \eqref{E:5.43g} that
\begin{align}\label{qo1}
\sup _{t\geq0}\left\|\bar z, \bar{\bf w}, \bar H\right\|_{W^{3,+\infty}(\mathbb R)} \lesssim \sup _{t\geq0}\left\|\bar z, \bar{\bf w}, \bar H\right\|_{H^4(\mathbb{R})}\lesssim \sup _{t\geq0}\left\|z, {\bf w}, H\right\|_{H^4(\mathbb{R})} \lesssim \nu
\end{align}
and 
\begin{align}\label{qo3}
	\|\bar z_t, \bar{\bf w}_t, \bar H_t\|_{W^{1,\infty}(\mathbb{R})}\lesssim\|\bar z_t, \bar{\bf w}_t, \bar H_t\|_{H^{2}(\mathbb{R})}\lesssim \|z_t, {\bf w}_t, H_t\|_{H^{2}(\Omega)}\lesssim \|z, {\bf w}, H_t\|_{H^{4}(\Omega)}\lesssim \nu.
\end{align}
Then it follows from Lemma \ref{de} that
\begin{align}\label{qo2}
\sup _{t\geq0}\left\|z^{\neq}, {\bf w}^{m d}, H^{\neq}, H^{\neq}_t\right\|_{W^{2,+\infty}(\Omega)} \lesssim \sup _{t\geq0}\left\|z^{\neq}, {\bf w}^{\neq}, H^{\neq}, H^{\neq}_t\right\|_{H^4(\Omega)} \lesssim \nu .
\end{align}
In addition, from Lemma \ref{l1} we have
\begin{align}
	\| z^{\neq}, {\bf w}^{\neq}, H^{\neq}, H^{\neq}_t\|_{L^2(\Omega)}^2\lesssim\|\nabla z^{\neq}, \nabla {\bf w}^{\neq}, \nabla H^{\neq}, \nabla H^{\neq}_t\|_{L^2(\Omega)}^2. \label{E:6.a}
\end{align}
Note that the zero mode of any function is independent of the transverse variables, $x^{\prime}=\left(x_2, x_3\right)$. Then from \eqref{E:5.3} we have
\begin{align}\label{E:6.1}
	\begin{cases}
		\partial_t z^{\neq}+\dive \left(\tilde \rho {\bf w}^{\neq}+{\bf J}_1\right)=0,\\
		\partial_t{\bf w}^{\neq}+\left(\frac{T}{\rho}\nabla z\right)^{\neq}-\left(\frac{1}{\rho}\Delta {\bf w}\right)^{\neq}+\nabla H^{\neq}=-{\bf h_1}^{\neq}-{\bf h_2}^{\neq},\\
		-\Delta H^{\neq}=z^{\neq}-e^{\tilde E}H^{\neq}-N_2^{\neq},
	\end{cases}
\end{align}
where 
\begin{align*}
&{\bf J}_1=\left(z^{\neq}{\bf w}^{\neq}\right)^{\neq}+\bar z{\bf w}^{\neq}+\bar {\bf w}z^{\neq}+\tilde {\bf u}z^{\neq}, \notag\\
	&{\bf h}_1^{\neq}=\tilde {\bf u}\cdot \nabla {\bf w}^{\neq}+\left({\bf w}^{\neq}\cdot \nabla {\bf w}^{\neq}\right)^{\neq}+{\bf w}^{\neq}(\nabla \bar{\bf w})+\bar {\bf w}\nabla {\bf w}^{\neq},\notag\\
	&{\bf h}_2^{\neq}=\frac{\Delta \tilde {\bf u}-T\nabla \tilde \rho}{\tilde \rho}\left(\frac{z}{\rho}\right)^{\neq}+\nabla \tilde {\bf u}\cdot {\bf w}^{\neq},\notag\\
&N_2^{\neq}=e^{\tilde E}(e^H-1-H)^{\neq}.
\end{align*}

Before we derive the decay rates of $\left(z^{\neq}, {\bf w}^{\neq}, H^{\neq}, H^{\neq}_t\right)$, we first establish some estimates of the nonlinear terms in \eqref{E:6.1}.
\begin{lemma}\label{L:6.1}
	Under the assumptions of Theorem \ref{theorem1.1}, it holds that
	\begin{align}
		&\|	{\bf J}_1\|_{H^1(\Omega)}\lesssim (\delta+\nu)\| \nabla z^{\neq}, \nabla {\bf w}^{\neq}\|_{L^2(\Omega)},\label{E:6.2b}\\
		&\|{\bf h}_1^{\neq}\|_{L^2(\Omega)}\lesssim (\delta+\nu)\| \nabla {\bf w}^{\neq}\|_{L^2(\Omega)},\quad \|{\bf h}_1^{\neq}\|_{H^1(\Omega)}\lesssim (\delta+\nu)\| \nabla {\bf w}^{\neq}, \nabla^2 {\bf w}^{\neq}\|_{L^2(\Omega)},\label{E:6.3b}\\
&\|{\bf h}_2^{\neq}\|_{H^1(\Omega)}\lesssim \delta	\left\|\nabla z^{\neq}, \nabla {\bf w}^{\neq}\right\|_{L^2(\Omega)},\quad \left\|{\bf h}_2^{\neq}\right\|_{H^2(\Omega)}\lesssim \delta \left\|\nabla z^{\neq}, \nabla {\bf w}^{\neq}\right\|_{H^1(\Omega)},\label{E:6.2a}\\
&\|N_{2}^{\neq}\|_{H^1(\Omega)}\lesssim (\delta+\nu)	\left\|\nabla H^{\neq}\right\|_{L^2(\Omega)},\quad \|N_{2}^{\neq}\|_{H^2(\Omega)}\lesssim (\delta+\nu)	\left\|\nabla H^{\neq}\right\|_{H^1(\Omega)},\label{E:6.6a}\\
	&\|	N_{2t}^{\neq}\|_{L^2(\Omega)}\lesssim (\delta+\nu)	\left\|\nabla H^{\neq}, \nabla H_t^{\neq}\right\|_{L^2(\Omega)}.\label{E:6.6}
	\end{align}
\end{lemma}
\begin{proof}
By the direct computations and using \eqref{qo1} and \eqref{qo2} we have \eqref{E:6.2b}-\eqref{E:6.3b}. To estimate the term ${\bf h}_2^{\neq}$, we rewrite $z^{\neq}$ as
\begin{align}\label{E:6.3}
	z^{\neq}=\left(\frac{z}{\rho}\rho\right)^{\neq}= \left[\left(\frac{z}{\rho}\right)^{\neq}z^{\neq}\right]^{\neq}+	\bar z\left(\frac{z}{\rho}\right)^{\neq}+\tilde \rho \left(\frac{z}{\rho}\right)^{\neq}+\overline{\left(\frac{z}{\rho}\right)}z^{\neq}.
\end{align}
Using the fact that $|\bar z|_{L^{\infty}}+\Big|\overline{\left(\frac{z}{\rho}\right)}\Big|_{L^{\infty}}+\Big|\left(\frac{z}{\rho}\right)^{\neq}\Big|_{L^{\infty}}\lesssim |z|_{L^{\infty}}+\Big|\frac{z}{\rho}\Big|_{L^{\infty}}\lesssim \nu$, we get
\begin{align}
	\left\|\left(\frac{z}{\rho}\right)^{\neq}\right\|_{L^2(\Omega)}\lesssim& \left\|\left[\left(\frac{z}{\rho}\right)^{\neq}z^{\neq}\right]^{\neq}\right\|_{L^2(\Omega)}+	\left\|z^{\neq}\right\|_{L^2(\Omega)}+	\left\|\overline{\left(\frac{z}{\rho}\right)}z^{\neq}\right\|_{L^2(\Omega)}
	\notag\\
	\lesssim& \left\|z^{\neq}\right\|_{L^2(\Omega)}\lesssim \left\|\nabla z^{\neq}\right\|_{L^2(\Omega)}.\nonumber
\end{align}
Furthermore, taking the \eqref{E:6.3} the gradient $\nabla$ and $\nabla^2$ respectively, we use the similar argument to get
\begin{align}
\left\|\left(\frac{z}{\rho}\right)^{\neq}\right\|_{H^1(\Omega)}\lesssim 	\left\|\nabla z^{\neq}\right\|_{L^2(\Omega)}
,\quad \left\|\nabla^2\left(\frac{z}{\rho}\right)^{\neq}\right\|_{L^2(\Omega)}\lesssim 	\left\|\nabla z^{\neq}, \nabla^2 z^{\neq}\right\|_{L^2(\Omega)},\label{E:6.2}
\end{align}
which implies \eqref{E:6.2a}.

To estimate the term $N_2^{\neq}$,  we use Poinca{\'r}e inequality, \eqref{qo1} and \eqref{qo2} to get
\begin{align}\label{E:6.7}
	&\|\left(e^H-1-H\right)^{\neq}\|_{L^2(\Omega)}
	\lesssim 	\|\nabla \left(e^H-1-H\right)^{\neq}\|_{L^2(\Omega)}= \| \left((e^H-1)\nabla H\right)^{\neq}\|_{L^2(\Omega)}
	\notag\\
	=&\| \left((e^H-1)^{\neq}\nabla H^{\neq}\right)^{\neq}+\overline{(e^H-1)}\nabla H^{\neq}+(\nabla \bar H)(e^H-1)^{\neq}\|_{L^2(\Omega)}
	\notag\\
	\lesssim&\left(\|\nabla H^{\neq}\|_{L^{\infty}(\Omega)}+\|\nabla \bar H\|_{L^{\infty}(\Omega)}\right)\|\left(e^H-1\right)^{\neq}\|_{L^2(\Omega)}+\|\overline{(e^H-1)}\|_{L^{\infty}(\Omega)}\|\nabla H^{\neq}\|_{L^2(\Omega)}
	\notag\\
	\lesssim&\nu\|\left(e^H-1\right)^{\neq}\|_{L^2(\Omega)}+\nu \|\nabla H^{\neq}\|_{L^2(\Omega)}
	\notag\\
	\lesssim&\nu\|\left(e^H-1-H\right)^{\neq}\|_{L^2(\Omega)}+\nu \|\nabla H^{\neq}\|_{L^2(\Omega)},
\end{align}
which yields that 
\begin{align}\label{E:6.8}
\|\left(e^H-1-H\right)^{\neq}\|_{L^2(\Omega)}\lesssim \nu \|\nabla H^{\neq}\|_{L^2(\Omega)}.
\end{align}
Moreover, it follows from \eqref{E:6.7}-\eqref{E:6.8} that 
\[\|\nabla\left(e^H-1-H\right)^{\neq}\|_{L^2(\Omega)}\lesssim \nu \|\nabla H^{\neq}\|_{L^2(\Omega)},\]
which implies the former of \eqref{E:6.6a}. And it is easy to check that the latter of \eqref{E:6.6a} holds by 
differentiaing $N_2^{\neq}$ twice. Finally, note that
\begin{align*}
N_{2t}^{\neq}=&e^{\tilde E}\left[\left((e^H-1)^{\neq}H_t^{\neq}\right)^{\neq}+(e^H-1)^{\neq}\bar H_t+\overline{(e^H-1)}H_t^{\neq}\right]
+\left(e^{\tilde E}\right)_t(e^H-1-H)^{\neq},
\end{align*}
we use \eqref{qo3} and \eqref{E:6.8} to get \eqref{E:6.6}.
Thus, the proof is completed.
\end{proof}
\subsection{the estimates of the non-zero mode of electric fields}
At the begining of this subsection, we first give the following lemma.
\begin{lemma}\label{L:6.2}
Under the assumptions of Theorem \ref{theorem1.1}, it holds that
	\begin{align}\label{E:6.9b}
	\|H_{tt}\|_{H^3(\Omega)}\lesssim \nu.
	\end{align}
\end{lemma}
\begin{proof}
Differentiating $\eqref{E:5.3}_3$ with respect to $t$ twice gives
\begin{align}\label{E:6.10b}
	-\Delta H_{tt}=z_{tt}-(e^{\tilde E}H)_{tt}-N_{2tt}.
\end{align}
Then multiplying \eqref{E:6.10b} by $H_{tt}$ as well as integrating the result over $\Omega$ gives that 
	\begin{align*}
	\|\nabla H_{tt}\|_{L^2(\Omega)}^2+\int_{\Omega}e^EH_{tt}^2dx=\int_{\Omega}z_{tt}H_{tt}dx-\int_{\Omega}\left((e^{\tilde E})_{tt}(e^H-1)+2(e^{\tilde E})_{t}e^HH_t+e^EH_t^2\right)H_{tt}dx.
	\end{align*}
	From $\eqref{E:5.3}_1$ and $\eqref{E:5.3}_2$ we get
		\begin{align}\label{bu5}
	\|z_{tt}\|_{H^1(\Omega)}^2\lesssim \|z, H\|_{H^3(\Omega)}^2+\|{\bf w}\|_{H^4(\Omega)}^2\lesssim \nu^2,
	\end{align}
together with \eqref{E:5.43g}, which yields that 
	\begin{align}\label{E:6.12b}
	\|H_{tt}\|_{H^1(\Omega)}\lesssim \nu.
\end{align}
Furthermore, substituting \eqref{bu5}-\eqref{E:6.12b} into \eqref{E:6.10b} leads to 
	\begin{align}
	\|\nabla^2 H_{tt}\|_{L^2(\Omega)}\lesssim \nu,\nonumber
\end{align}
where we have used the fact $\|\Delta H_{tt}\|_{L^2(\Omega)}\geq c_0\|\nabla^2 H_{tt}\|_{L^2(\Omega)}$.
Finally, we take the gradient $\nabla \eqref{E:6.10b}$ and use the \eqref{bu5}-\eqref{E:6.12b} to get
\[\|\nabla^3 H_{tt}\|_{L^2(\Omega)}\lesssim \nu.\]
 Thus, the proof is completed.
	\end{proof}
\begin{lemma}\label{L:6.3}
Under the assumptions of Theorem \ref{theorem1.1}, it holds that
	\begin{align}\label{E:6.27}
		&\frac{d}{dt}(E_1(t)+C_9E_2(t))+\|\nabla H^{\neq}, \nabla H_t^{\neq}\|_{H^1(\Omega)}^2
		\lesssim(\delta+\nu)\| \nabla z^{\neq}, \nabla {\bf w}^{\neq}, \nabla^2{\bf w}^{\neq}\|_{L^2(\Omega)}^2,
	\end{align}
where $E_1(t)$ and $E_2(t)$ are given as in \eqref{E:6.35c} and \eqref{E:6.35d} respectively.
\end{lemma}
\begin{proof}
	Multiplying $\eqref{E:6.1}_2$ by $\cdot \tilde \rho\nabla H^{\neq}$ and integrating the resultants over $\Omega$ gives that
	\begin{align}\label{E:6.9}
		&\int_{\Omega}\partial_t{\bf w}^{\neq}\cdot \tilde\rho\nabla H^{\neq}dx+\int_{\Omega}\left(\frac{T}{\rho}\nabla z\right)^{\neq}\cdot \tilde \rho \nabla H^{\neq}dx-\int_{\Omega}\left(\frac{1}{\rho}\Delta {\bf w}\right)^{\neq}\cdot \tilde \rho\nabla H^{\neq}dx+\int_{\Omega} \tilde \rho\left(\nabla H^{\neq}\right)^2dx\notag\\
		=&-\int_{\Omega}{\bf h_1}^{\neq}\cdot \tilde \rho\nabla H^{\neq}dx-\int_{\Omega}{\bf h_2}^{\neq}\cdot \tilde \rho\nabla H^{\neq}dx.
	\end{align}
Applying the same argements in \eqref{E:5.5b} we get
\begin{align*}
	&\int_{\Omega}\partial_t{\bf w}^{\neq}\cdot \tilde \rho\nabla H^{\neq}dx
	=\frac{d}{dt}E_{1,1}(t)+J_2,
\end{align*}	
where 
\begin{align}
&E_{1,1}(t)=\int_{\Omega}\left[\nabla H^{\neq}\cdot \nabla H_t^{\neq}+e^{\tilde E}H^{\neq}H_t^{\neq}+\left(e^{\tilde E}\right)_t(H^{\neq})^2+H^{\neq}N_{2t}^{\neq}-\nabla H^{\neq}\cdot {\bf J}_1\right]dx,\label{bu1}\\
&J_2=\int_{\Omega}\left[\left(\nabla H_t^{\neq}\right)^2+e^{\tilde E}(H_t^{\neq})^2+\left(e^{\tilde E}\right)_tH^{\neq}H_t^{\neq}+H_t^{\neq}N_{2t}^{\neq}-\nabla H_t^{\neq}\cdot {\bf J}_1+\partial_t\tilde \rho{\bf w}^{\neq}\cdot \nabla H^{\neq}\right]dx.\nonumber
\end{align}
It follows from \eqref{E:6.a}, \eqref{E:6.2b} and \eqref{E:6.6} that
\begin{align*}
	J_2\lesssim \|\nabla H_t^{\neq}\|_{L^2(\Omega)}^2+(\delta+\nu)\|\nabla H^{\neq}, \nabla z^{\neq}, \nabla {\bf w}^{\neq}\|_{L^2(\Omega)}^2.
\end{align*}
Thus, we get 
\begin{align}\label{E:6.10}
	&\int_{\Omega}\partial_t{\bf w}^{\neq}\cdot \tilde \rho\nabla H^{\neq}dx
	\gtrsim \frac{d}{dt}E_{1,1}(t)
-\|\nabla H_t^{\neq}\|_{L^2(\Omega)}^2-(\delta+\nu)\|\nabla H^{\neq}, \nabla z^{\neq}, \nabla {\bf w}^{\neq}\|_{L^2(\Omega)}^2.
\end{align}	
Next, the direct computations shows that
\begin{align*}
	&\int_{\Omega}\left(\frac{T}{\rho}\nabla z\right)^{\neq}\cdot \tilde \rho \nabla H^{\neq}dx
	=\int_{\Omega}T\nabla z^{\neq}\cdot \nabla H^{\neq}dx-\int_{\Omega}T\left(\frac{z}{\rho}\nabla z\right)^{\neq}\cdot \nabla H^{\neq}dx\notag\\
	=&\int_{\Omega}T\left[\left(\Delta H^{\neq}\right)^2+e^{\tilde E}\left(\nabla H^{\neq}\right)^2\right]dx-\int_{\Omega}T\left(\frac{z}{\rho}\nabla z\right)^{\neq}\cdot \nabla H^{\neq}dx+J_3,
\end{align*}
where 
\begin{align*}
	J_3&=T\int_{\Omega}\nabla(e^{\tilde E})H^{\neq}\cdot \nabla H^{\neq}dx+T\int_{\Omega}\nabla N_2^{\neq}\cdot \nabla H^{\neq}dx.
\end{align*}
Note that
\begin{align}\label{po1m}
&\left(\frac{z}{\rho}\nabla z\right)^{\neq}=	\left(\left(\frac{z}{\rho}\right)^{\neq}\nabla z^{\neq}\right)^{\neq}+ \overline{\left(\frac{z}{\rho}\right)}\nabla z^{\neq}+(\nabla z)^{od}\left(\frac{z}{\rho}\right)^{\neq},
\end{align}
then we use \eqref{E:6.6a}, \eqref{E:6.2} and the fact that $\|\Delta H^{\neq}\|_{L^2(\Omega)}\geq c_0\|\nabla^2 H^{\neq}\|_{L^2(\Omega)}$ to get
\begin{align}\label{E:6.11}
	\int_{\Omega}\left(\frac{T}{\rho}\nabla z\right)^{\neq}\cdot \tilde \rho \nabla H^{\neq}dx
	\gtrsim \|\nabla H^{\neq}\|_{H^1(\Omega)}^2-(\delta+\nu) \|\nabla z^{\neq}\|_{L^2(\Omega)}^2.
\end{align}
In addition, for the third term on the left hand side of \eqref{E:6.9}, we have
\begin{align*}
-\left(\frac{1}{\rho}\Delta {\bf w}\right)^{\neq}\cdot \tilde \rho\nabla H^{\neq}
=&-\Delta {\bf w}^{\neq}\cdot \nabla H^{\neq}+\left(\frac{z}{\rho}\Delta {\bf w}\right)^{\neq}\cdot \nabla H^{\neq}
\notag\\
=&\dive(\cdots)+\dive{\bf w}^{\neq}\Delta H^{\neq}+\left(\frac{z}{\rho}\Delta {\bf w}\right)^{\neq}\cdot \nabla H^{\neq},
\end{align*}
where $(\cdots)=\mbox{curl}{\bf w}^{\neq}\times \nabla H^{\neq}-\dive {\bf w}^{\neq}\nabla H^{\neq}$.
It follows from $\eqref{E:6.1}_1$ and $\eqref{E:6.1}_3$ that 
\begin{align}
\dive{\bf w}^{\neq}\Delta H^{\neq}=&-\frac{1}{\tilde \rho}\left(z^{\neq}_t+\nabla \tilde \rho \cdot {\bf w}^{\neq}+\dive {\bf J}_1\right)\Delta H^{\neq}
\notag\\
=&-\frac{1}{\tilde \rho}\left[\left(-\Delta H^{\neq}+e^{\tilde E}H^{\neq}+N_2^{\neq}\right)_t+\nabla \tilde \rho \cdot {\bf w}^{\neq}+\dive {\bf J}_1\right]\Delta H^{\neq},\nonumber
 \end{align}
 together with \eqref{E:6.2b} and \eqref{E:6.6}, which leads to 
 \begin{align}\label{bu2}
 \int_{\Omega} \dive{\bf w}^{\neq}\Delta H^{\neq}dx\gtrsim& \frac{d}{dt} E_{1,2}(t)
-(\delta+\nu)\|\nabla H^{\neq}, \nabla^2 H^{\neq}, \nabla H_t^{\neq}, \nabla z^{\neq}, \nabla {\bf w}^{\neq}\|_{L^2(\Omega)}^2,
 \end{align}
where 
\begin{align}\label{bu4}
E_{1,2}(t)=\int_{\Omega}\frac{1}{2\tilde \rho}\left[(\Delta H^{\neq})^2+e^{\tilde E}(\nabla H^{\neq})^2\right]dx.
\end{align}
On the other hand, note that 
\begin{align*}
	&\left(\frac{z}{\rho}\Delta {\bf w}\right)^{\neq}=	\left(\left(\frac{z}{\rho}\right)^{\neq}\Delta {\bf w}^{\neq}\right)^{\neq}+\overline{\left(\frac{z}{\rho}\right)}\Delta {\bf w}^{\neq}+\overline{(\Delta {\bf w})}\left(\frac{z}{\rho}\right)^{\neq},
\end{align*}
then we use \eqref{E:6.2} to get
\begin{align}\label{bu3}
\Big\|\left(\frac{z}{\rho}\Delta {\bf w}\right)^{\neq}\cdot \nabla H^{\neq}\Big\|_{L^1(\Omega)}\lesssim(\delta+\nu)\|\nabla H^{\neq}, \nabla z^{\neq}, \nabla^2 {\bf w}^{\neq}\|_{L^2(\Omega)}^2.
\end{align}
Combining \eqref{bu2} with \eqref{bu3} we get 
\begin{align}\label{E:6.14}
	&-\int_{\Omega}\left(\frac{1}{\rho}\Delta {\bf w}\right)^{\neq}\cdot \tilde \rho\nabla H^{\neq}dx
	\gtrsim \frac{d}{dt}E_{1,2}(t)-(\delta+\nu)\|\nabla H^{\neq}, \nabla^2 H^{\neq}, \nabla H_t^{\neq}, \nabla z^{\neq}, \nabla {\bf w}^{\neq}, \nabla^2 {\bf w}^{\neq}\|_{L^2(\Omega)}^2.
\end{align}
Moreover, it follows from \eqref{E:6.3b}-\eqref{E:6.2a} that
\begin{align}\label{E:6.15}
-\int_{\Omega}{\bf h_1}^{\neq}\cdot \tilde \rho\nabla H^{\neq}dx-\int_{\Omega}{\bf h_2}^{\neq}\cdot \tilde \rho\nabla H^{\neq}dx
\lesssim(\delta+\nu)\|\nabla H^{\neq}, \nabla z^{\neq}, \nabla {\bf w}^{\neq}\|_{L^2(\Omega)}^2.
\end{align}
Substituting \eqref{E:6.10}-\eqref{E:6.11} and \eqref{E:6.14}-\eqref{E:6.15} into \eqref{E:6.9} gives that
\begin{align}\label{E:6.16}
&\frac{d}{dt}E_1(t)
+\|\nabla H^{\neq}\|_{H^1(\Omega)}^2\lesssim\|\nabla H_t^{\neq}\|_{L^2(\Omega)}^2+(\delta+\nu)\|\nabla z^{\neq}, \nabla {\bf w}^{\neq}, \nabla^2 {\bf w}^{\neq}\|_{L^2(\Omega)}^2,
\end{align}
where 
\begin{align}\label{E:6.35c}
E_1(t)=E_{1,1}(t)+E_{1,2}(t),
\end{align}
and  $E_{1,1}(t), E_{1,2}(t)$ are defined in \eqref{bu1} and \eqref{bu4} respectively.

To control the term containing the electric potential on the right hand side of \eqref{E:6.16}, multiplying $\eqref{E:6.1}_2$ by $\cdot \tilde \rho\nabla H_t^{\neq}$ and integrating the resultant over $\Omega$, we have 
	\begin{align}\label{E:6.21}
	&\frac{d}{dt}\int_{\Omega} \frac{\tilde \rho}{2}\left(\nabla H^{\neq}\right)^2dx+\int_{\Omega}\partial_t{\bf w}^{\neq}\cdot \tilde\rho\nabla H_t^{\neq}dx+\int_{\Omega}\left(\frac{T\tilde \rho}{\rho}\nabla z\right)^{\neq}\cdot  \nabla H_t^{\neq}dx-\int_{\Omega}\left(\frac{\tilde \rho}{\rho}\Delta {\bf w}\right)^{\neq}\cdot \nabla H_t^{\neq}dx
	\notag\\
	&=\int_{\Omega} \frac{\tilde \rho_t}{2}\left(\nabla H^{\neq}\right)^2dx-\int_{\Omega}{\bf h_1}^{\neq}\cdot \tilde \rho\nabla H_t^{\neq}dx-\int_{\Omega}{\bf h_2}^{\neq}\cdot \tilde \rho\nabla H_t^{\neq}dx.
\end{align}
 Applying the similar argument in \eqref{E:6.10} and using Lemma \ref{L:6.1}, we get
 \begin{align}\label{E:6.32}
 \int_{\Omega}\partial_t{\bf w}^{\neq}\cdot \tilde\rho\nabla H_t^{\neq}dx\gtrsim&
\frac{d}{dt} \int_{\Omega}\left[\frac{e^{\tilde E}}{2}\left(H_t^{\neq}\right)^2+\frac{1}{2}\left(\nabla H_t^{\neq}\right)^2\right]dx+\int_{\Omega}N_{2tt}^{\neq}H_t^{\neq}dx
 \notag\\&
 -(\delta+\nu)\|\nabla H_t^{\neq}, \nabla z^{\neq}, \nabla {\bf w}^{\neq}\|_{L^2(\Omega)}^2.
 \end{align}
By the careful computation we have
\begin{align}\label{bu7}
	N_{2tt}^{\neq}H_t^{\neq}=\left[\frac{1}{2}e^{\tilde E}\overline{(e^H-1)}\left(H_t^{\neq}\right)^2\right]_t+\left(e^{\tilde E}\right)_{tt}(e^H-1-H)^{\neq}H_t^{\neq}+J_4+J_5,
\end{align}
where 
\begin{align*}
J_4=&2\left(e^{\tilde E}\right)_{t}\left[\left((e^H-1)^{\neq}H^{\neq}_t\right)^{\neq}+\overline{(e^H-1)}H^{\neq}_t+(e^H-1)^{\neq}\bar H_t\right]H_t^{\neq},\notag\\
J_5=&e^{\tilde E}\left[\left((e^H-1)^{\neq}(H^2_t)^{\neq}\right)^{\neq}+(e^H-1)^{\neq}\overline{(H^2_t)}+\overline{(e^H)}(H^2_t)^{\neq}\right]H_t^{\neq}
\notag\\
&+e^{\tilde E}\left[\left((e^H-1)^{\neq}H_{tt}^{\neq}\right)^{\neq}+(e^H-1)^{\neq}\bar H_{tt}\right]H_t^{\neq}-\frac{1}{2}\left(e^{\tilde E}\overline{(e^H-1)}\right)_t\left(H_t^{\neq}\right)^2.
\end{align*}
From \eqref{E:6.9b} we get
\[\|H_{tt}^{\neq}\|_{W^{1,\infty}(\Omega)}+\|\bar H_{tt}\|_{W^{1,\infty}(\mathbb{R})}\lesssim \|H_{tt}^{\neq}\|_{H^{3}(\Omega)}+\|\bar H_{tt}\|_{H^{2}(\mathbb{R})}\lesssim \|H_{tt}\|_{H^{3}(\Omega)}\lesssim \nu,\]
together with \eqref{E:6.8} and the fact $(H^2_t)^{\neq}=\left((H^{\neq}_t)^2\right)^{\neq}+2\bar H_tH_t^{\neq}$, which leads to 
\begin{align}\label{E:6.34}
\|J_4\|_{L^1(\Omega)}+	\|J_5\|_{L^1(\Omega)}\lesssim (\delta+\nu)\| \nabla H^{\neq}, \nabla H_t^{\neq}\|_{L^2(\Omega)}^2.	
\end{align}
Substituting \eqref{bu7} and \eqref{E:6.34} into \eqref{E:6.32}, we get
 \begin{align}\label{E:6.35}
	\int_{\Omega}\partial_t{\bf w}^{\neq}\cdot \tilde\rho\nabla H_t^{\neq}dx\gtrsim& \frac{d}{dt} E_{2,1}(t)
	-(\delta+\nu)\| \nabla H^{\neq}, \nabla H_t^{\neq}, \nabla z^{\neq}, \nabla {\bf w}^{\neq}\|_{L^2(\Omega)}^2,
\end{align}
 where 
 \begin{align}
 E_{2,1}(t)=\int_{\Omega}\frac{1}{2}\left[\overline{\left(e^{ E}\right)}\left(H_t^{\neq}\right)^2+\left(\nabla H_t^{\neq}\right)^2\right]dx.\nonumber
 \end{align}
Taking the same tecnique used in \eqref{E:6.11} and \eqref{E:6.14} we have 
\begin{align}\label{E:6.14a}
	&\int_{\Omega}\left(\frac{T\tilde \rho}{\rho}\nabla z\right)^{\neq}\cdot  \nabla H_t^{\neq}dx
	=\int_{\Omega}T\nabla z^{\neq}\cdot \nabla H_t^{\neq}dx-\int_{\Omega}T\left(\frac{z}{\rho}\nabla z\right)^{\neq}\cdot \nabla H_t^{\neq}dx\notag\\
	\gtrsim& \frac{d}{dt}\int_{\Omega}\frac{T}{2}\left[e^{\tilde E}\left(\nabla H^{\neq}\right)^2+\left(\Delta H^{\neq}\right)^2\right]dx
	-(\delta+\nu)\|\nabla H^{\neq}, \nabla H_t^{\neq}, \nabla z^{\neq}\|_{L^2(\Omega)}^2
\end{align}
and 
\begin{align}
	-\int_{\Omega}\left(\frac{1}{\rho}\Delta {\bf w}\right)^{\neq}\cdot \tilde \rho\nabla H_t^{\neq}dx
	\gtrsim&\|\nabla H_t^{\neq}\|_{H^1(\Omega)}^2
-(\delta+\nu)\|\nabla H^{\neq}, \nabla z^{\neq}, \nabla {\bf w}^{\neq}, \nabla^2 {\bf w}^{\neq}\|_{L^2(\Omega)}^2,
\end{align}
where we have used the fact that $\|\Delta H_t^{\neq}\|_{L^2(\Omega)}\geq c_0\|\nabla^2 H_t^{\neq}\|_{L^2(\Omega)}$. In addition, it follows from Lemma \ref{L:6.1} that
\begin{align}\label{E:6.15a}
	-\int_{\Omega}{\bf h_1}^{\neq}\cdot \tilde \rho\nabla H_t^{\neq}dx-\int_{\Omega}{\bf h_2}^{\neq}\cdot \tilde \rho\nabla H_t^{\neq}dx
	\lesssim(\delta+\nu)\|\nabla H_t^{\neq}, \nabla z^{\neq}, \nabla {\bf w}^{\neq}\|_{L^2(\Omega)}^2.
\end{align}
Substituting \eqref{E:6.35} and \eqref{E:6.14a}-\eqref{E:6.15a} into \eqref{E:6.21} gives that 
 \begin{align}\label{E:6.35a}
	 \frac{d}{dt}E_2(t)+\|\nabla H_t^{\neq}\|_{H^1(\Omega)}^2
	\lesssim
	(\delta+\nu)\|\nabla H^{\neq}, \nabla z^{\neq}, \nabla {\bf w}^{\neq}, \nabla^2 {\bf w}^{\neq}\|_{L^2(\Omega)}^2,
\end{align}
where 
\begin{align}\label{E:6.35d}
	E_2(t)=E_{2,1}(t)+\int_{\Omega}\frac{T}{2}\left[e^{\tilde E}\left(\nabla H^{\neq}\right)^2+\left(\Delta H^{\neq}\right)^2\right]dx.
\end{align}
Finally, taking the procedure as $\eqref{E:6.16}+C_9\eqref{E:6.35a}$ with some positive constant $C_9$ large enough leads to \eqref{E:6.27}.
Thus, the proof is completed.
\end{proof}
\subsection{the estimates of non-zero mode of $z^{\neq}, {\bf w}^{\neq}$}
\begin{lemma}\label{L:6.4}
Under the assumptions of Theorem \ref{theorem1.1}, it holds that
		\begin{align}\label{E:6.49}
	&\frac{d}{dt}\int_{\Omega}\left(\frac12({\bf w}^{\neq})^2+\frac{T}{2\tilde \rho}(z^{\neq})^2\right)dx+\|\nabla{\bf w}^{\neq}\|_{L^2(\Omega)}^2
	\lesssim(\delta+\nu) \|\nabla z^{\neq}\|_{L^2(\Omega)}^2+\|\nabla H^{\neq}\|_{L^2(\Omega)}^2.
	\end{align}
\end{lemma}
\begin{proof}
Multiplying $\eqref{E:6.1}_2$ by $\cdot {\bf w}^{\neq}$ and integrating the resultant over $\Omega$, we use \eqref{E:6.3b}-\eqref{E:6.2a} to get 
	\begin{align}\label{E:6.46}
		&\frac{d}{dt}\int_{\Omega}\frac12({\bf w}^{\neq})^2dx+\int_{\Omega}\left(\frac{T}{\rho}\nabla z\right)^{\neq}\cdot {\bf w}^{\neq}dx-\int_{\Omega}\left(\frac{1}{\rho}\Delta {\bf w}\right)^{\neq}\cdot {\bf w}^{\neq}dx
		\notag\\
		\lesssim&(\delta+\nu) \|\nabla z^{\neq}\|_{L^2(\Omega)}^2+(\epsilon+\delta+\nu) \|\nabla{\bf w}^{\neq}\|_{L^2(\Omega)}^2+C_{\epsilon}\|\nabla H^{\neq}\|_{L^2(\Omega)}^2,
	\end{align}
where $\epsilon$ is a small constant to be determined later. It follows from $\eqref{E:6.1}_1$ and \eqref{E:6.2} that 
\begin{align}\label{E:6.47}
\int_{\Omega}\left(\frac{T}{\rho}\nabla z\right)^{\neq}\cdot {\bf w}^{\neq}dx=&\int_{\Omega}\frac{T}{\tilde \rho}\nabla z^{\neq}\cdot {\bf w}^{\neq}dx-\int_{\Omega}\left(\frac{Tz}{\rho}\nabla z\right)^{\neq}\cdot {\bf w}^{\neq}dx
\notag\\
=&-\int_{\Omega}\nabla\left(\frac{T}{\tilde \rho}\right) z^{\neq}\cdot {\bf w}^{\neq}dx+\int_{\Omega}\frac{Tz^{\neq}}{\tilde \rho}\left(\partial_t z^{\neq}+\nabla \tilde \rho\cdot {\bf w}^{\neq}\right)dx
\notag\\
&-\int_{\Omega}\nabla \left(\frac{Tz^{\neq}}{\tilde \rho}\right)\cdot {\bf J}_1dx-\int_{\Omega}\left(\frac{Tz}{\rho}\nabla z\right)^{\neq}\cdot {\bf w}^{\neq}dx
\notag\\
\gtrsim&\frac{d}{dt}\int_{\Omega}\frac{T}{2\tilde \rho}(z^{\neq})^2dx-(\delta+\nu)\|\nabla z^{\neq}, \nabla{\bf w}^{\neq}\|_{L^2(\Omega)}^2
\end{align}
and 
\begin{align}\label{E:6.48}
-\int_{\Omega}\left(\frac{1}{\rho}\Delta {\bf w}\right)^{\neq}\cdot {\bf w}^{\neq}dx=&-\int_{\Omega}\frac{1}{\tilde \rho}\Delta {\bf w}^{\neq}\cdot {\bf w}^{\neq}dx+\int_{\Omega}\left(\frac{z}{\rho}\Delta {\bf w}\right)^{\neq}\cdot {\bf w}^{\neq}dx
\notag\\
=&\int_{\Omega}\frac{1}{\tilde \rho}\left(\nabla {\bf w}^{\neq}\right)^2dx+\int_{\Omega}\nabla\left(\frac{1}{\tilde \rho}\right)\cdot\nabla {\bf w}^{\neq}\cdot {\bf w}^{\neq}dx
\notag\\
&-\int_{\Omega}\left(\frac{z}{ \rho}\nabla {\bf w}\right)^{\neq}\cdot \nabla{\bf w}^{\neq}dx-\int_{\Omega}\left(\nabla\left(\frac{z}{\rho}\right)\cdot \nabla{\bf w}\right)^{\neq}\cdot {\bf w}^{\neq}dx
\notag\\
\gtrsim&\|\nabla{\bf w}^{\neq}\|_{L^2(\Omega)}^2.
\end{align}
Then, substituting \eqref{E:6.47}-\eqref{E:6.48} into \eqref{E:6.46} and choosing $\epsilon$ small enough yields that \eqref{E:6.49}. Thus, the proof is completed.
\end{proof}

\begin{lemma}\label{L:6.5}
Under the assumptions of Theorem \ref{theorem1.1}, it holds that
		\begin{align}\label{E:6.61}
		&\frac{d}{dt}\int_{\Omega}\left[C_{10}\left(\frac12(\nabla{\bf w}^{\neq})^2+\frac{T}{2\tilde \rho^2}\left(\nabla z^{\neq}\right)^2\right)-{\bf w}^{\neq}\cdot \nabla z^{\neq}\right]dx+\|\nabla z^{\neq}, \nabla^2 {\bf w}^{\neq}\|_{L^{2}(\Omega)}^2
		\lesssim \|\nabla {\bf w}^{\neq}, \nabla H^{\neq}\|_{L^2(\Omega)}^2.
	\end{align}
\end{lemma}
\begin{proof}
		Taking the procedure as $\int_{\Omega}\partial_x^{\alpha}\eqref{E:6.1}_2\cdot \partial_x^{\alpha}{\bf w}^{\neq}dx$ with the multi-index $|\alpha|=1$, we have
			\begin{align}\label{E:6.50}
		&\frac{d}{dt}\int_{\Omega}\frac12(\partial_x^{\alpha}{\bf w}^{\neq})^2dx+\int_{\Omega}\partial_x^{\alpha}\left(\frac{T}{\rho}\nabla z\right)^{\neq}\cdot \partial_x^{\alpha}{\bf w}^{\neq}dx-\int_{\Omega}\partial_x^{\alpha}\left(\frac{1}{\rho}\Delta {\bf w}\right)^{\neq}\cdot \partial_x^{\alpha}{\bf w}^{\neq}dx
		\notag\\
		\lesssim&(\delta+\nu)\|\nabla z^{\neq}, \nabla {\bf w}^{\neq}, \nabla^{2} {\bf w}^{\neq}\|_{L^2(\Omega)}^2+\epsilon\|\nabla \partial_x^{\alpha} {\bf w}^{\neq}\|_{L^2(\Omega)}^2+C_{\epsilon}\|\partial_x^{\alpha}H^{\neq}\|_{L^2(\Omega)}^2.
		\end{align}
	It follows that
\begin{align}
	&\int_{\Omega}\partial_x^{\alpha}\left(\frac{T}{\rho}\nabla z\right)^{\neq}\cdot \partial_x^{\alpha}{\bf w}^{\neq}dx
	=-\int_{\Omega}\frac{T}{\tilde \rho}\partial_x^{\alpha}z^{\neq}\partial_x^{\alpha}\dive{\bf w}^{\neq}dx+J_6,\nonumber
\end{align}	
where 
\begin{align*}
J_6=&-\int_{\Omega}\nabla\left(\frac{T}{\tilde \rho}\right)\partial_x^{\alpha}z^{\neq}\cdot\partial_x^{\alpha}{\bf w}^{\neq}dx+\sum_{|\beta|=1, \beta \leq \alpha}\int_{\Omega}\partial_x^{\beta}\left(\frac{T}{\tilde \rho}\right)\partial_x^{\alpha-\beta}\nabla z^{\neq}\cdot \partial_x^{\alpha}{\bf w}^{\neq}dx
\notag\\
&-\frac{T}{\tilde \rho}\int_{\Omega}\partial_x^{\alpha}\left(\frac{z}{\rho}\nabla z\right)^{\neq}\cdot \partial_x^{\alpha}{\bf w}^{\neq}dx-\sum_{|\beta|=1, \beta \leq \alpha}\int_{\Omega}\partial_x^{\beta}\left(\frac{T}{\tilde \rho}\right)\partial_x^{\alpha-\beta}\left(\frac{z}{\rho}\nabla z\right)^{\neq}\cdot \partial_x^{\alpha}{\bf w}^{\neq}dx.
\end{align*}
Similar to \eqref{E:5.33}, it follows from $\eqref{E:6.1}_1$ that 
\begin{align}\label{E:6.52}
\partial_x^\alpha(\operatorname{div} {\bf w}^{\neq})=\frac{-1}{\tilde \rho}\left\{\partial_x^\alpha z^{\neq}_t+\partial_x^\alpha\left(\nabla \tilde \rho \cdot {\bf w}^{\neq}\right)+\partial_x^\alpha\dive {\bf J}_1+\sum_{|\beta|=1, \beta \leq \alpha} \partial_x^\beta \tilde\rho \partial_x^{\alpha-\beta}(\operatorname{div} {\bf w}^{\neq})\right\}.
\end{align}
We can check that 
\begin{align}
\int_{\Omega}\frac{T}{\tilde \rho^2}\partial_x^{\alpha}z^{\neq}\partial_x^{\alpha}z_t^{\neq}dx\gtrsim \frac{d}{dt}\int_{\Omega}\frac{T}{2\tilde \rho^2}\left(\partial_x^{\alpha}z^{\neq}\right)^2dx-\delta \|\partial_x^{\alpha} z^{\neq}\|_{L^2(\Omega)}^2
\end{align}
and from \eqref{qo1} and \eqref{qo2} that
\begin{align}\label{E:6.54}
	&\int_{\Omega}\frac{T}{\tilde \rho^2}\partial_x^{\alpha}z^{\neq}\partial_x^{\alpha}\dive {\bf J}_1dx
	\notag\\
=&\int_{\Omega}\frac{T}{\tilde \rho^2}\partial_x^{\alpha}z^{\neq}\partial_x^{\alpha}\dive \left(\left(z^{\neq}{\bf w}^{\neq}\right)^{\neq}+\bar z{\bf w}^{\neq}\right)dx-\int_{\Omega}\dive\left(\frac{T}{2\tilde \rho^2}\left(\bar{\bf w}+\tilde {\bf u}\right)\right)(\partial_x^{\alpha}z^{\neq})^2dx
\notag\\
&+\int_{\Omega}\frac{T}{\tilde \rho^2}\partial_x^{\alpha}z^{\neq}\partial_x^{\alpha}\left[\dive\left(\bar{\bf w}+\tilde {\bf u}\right)z^{\neq}\right]dx+\sum_{|\beta|=1, \beta \leq \alpha}\int_{\Omega}\frac{T}{\tilde \rho^2}\partial_x^{\alpha}z^{\neq}\partial_x^{\beta}\left(\bar{\bf w}+\tilde {\bf u}\right)\partial_x^{\alpha-\beta}\nabla z^{\neq}dx
\notag\\
	\lesssim&\|{\bf w}^{\neq}\|_{L^{\infty}(\Omega)}\|\partial_x^{\alpha}\nabla z^{\neq}\|_{L^{2}(\Omega)}\|\partial_x^{\alpha}z^{\neq}\|_{L^{2}(\Omega)}
	\notag\\
	&+(\delta+\|z^{\neq}\|_{W^{2,\infty}(\Omega)}+\|\partial_x^{\alpha+1}\bar{\bf w},\partial_x^{\alpha+1}\bar z\|_{L^{\infty}(\mathbb{R})})\|\nabla z^{\neq}, \partial_x^{\alpha}z^{\neq}, \nabla{\bf w}^{\neq}, \partial_x^{\alpha}{\bf w}^{\neq},\partial_x^{\alpha}\dive {\bf w}^{\neq}\|_{L^{2}(\Omega)}^2
	\notag\\
	\lesssim&(\delta+\nu)\|\nabla z^{\neq}, \nabla^{|\alpha|}z^{\neq}, \nabla{\bf w}^{\neq}, \nabla^2{\bf w}^{\neq},\nabla^{|\alpha|+1} {\bf w}^{\neq}\|_{L^{2}(\Omega)}^2,
\end{align}
where in the last inequality we have used the fact that 
\[\|\partial_x^{\alpha}\nabla z^{\neq}\|_{L^{2}(\Omega)}\lesssim \|z^{\neq}\|_{H^{3}(\Omega)}\lesssim \nu\quad\mbox{and}\quad \|{\bf w}^{\neq}\|_{L^{\infty}(\Omega)}\lesssim \|{\bf w}^{\neq}\|_{H^{2}(\Omega)}.\]
Then  \eqref{E:6.52}-\eqref{E:6.54} gives that
\begin{align}\label{aE:6.54}
-\int_{\Omega}\frac{T}{\tilde \rho}\partial_x^{\alpha}z^{\neq}\partial_x^{\alpha}\dive{\bf w}^{\neq}dx\gtrsim \frac{d}{dt}\int_{\Omega}\frac{T}{2\tilde \rho^2}\left(\partial_x^{\alpha}z^{\neq}\right)^2dx-(\delta+\nu)\|\nabla z^{\neq}, \nabla{\bf w}^{\neq}, \nabla^2 {\bf w}^{\neq}\|_{L^{2}(\Omega)}^2.
\end{align}
In addition, it follows from \eqref{qo2}, \eqref{E:6.2} and \eqref{po1m} that
\begin{align}
	\|J_6\|_{L^1(\Omega)}\lesssim&(\delta+\|z^{\neq}\|_{W^{2,\infty}(\Omega)})\|\partial_x^{\alpha}z^{\neq}, \nabla z^{\neq}, \partial_x^{\alpha}{\bf w}^{\neq}\|_{L^{2}(\Omega)}^2+\|z^{\neq}\|_{L^{\infty}(\Omega)}\|\partial_x^{\alpha}\nabla z^{\neq}\|_{L^{2}(\Omega)}\|\partial_x^{\alpha}{\bf w}^{\neq}\|_{L^{2}(\Omega)}
	\notag\\
	\lesssim&(\delta+\nu)\|\nabla z^{\neq}, \nabla^{|\alpha|} z^{\neq},  \nabla^{|\alpha|} {\bf w}^{\neq}\|_{L^{2}(\Omega)}^2,\nonumber
	\end{align}
together with \eqref{aE:6.54}, which yields that 
\begin{align}\label{E:6.56}
	&\int_{\Omega}\partial_x^{\alpha}\left(\frac{T}{\rho}\nabla z\right)^{\neq}\cdot \partial_x^{\alpha}{\bf w}^{\neq}dx
	\notag\\
	\gtrsim& \frac{d}{dt}\int_{\Omega}\frac{T}{2\tilde \rho^2}\left(\partial_x^{\alpha}z^{\neq}\right)^2dx
	-(\delta+\nu)\|\nabla z^{\neq}, \nabla^{|\alpha|} z^{\neq}, \nabla{\bf w}^{\neq}, \nabla^{2} {\bf w}^{\neq}, \nabla^{|\alpha|+1} {\bf w}^{\neq}\|_{L^{2}(\Omega)}^2.
\end{align}	
Similarly, it holds that
\begin{align*}
	-\int_{\Omega}\partial_x^{\alpha}\left(\frac{1}{\rho}\Delta {\bf w}\right)^{\neq}\cdot \partial_x^{\alpha}{\bf w}^{\neq}dx=\int_{\Omega}\frac{1}{\tilde \rho} |\nabla\partial_x^{\alpha}{\bf w}^{\neq}|^2dx+J_7,
\end{align*}
where 
\begin{align*}
	J_7=&\int_{\Omega}\nabla\left(\frac{1}{\tilde \rho}\right)\cdot(\nabla \partial_x^{\alpha}{\bf w}^{\neq}\cdot  \partial_x^{\alpha}{\bf w}^{\neq})dx
-\sum_{|\beta|=1, \beta \leq \alpha}\int_{\Omega}\partial_x^{\beta}\left(\frac{1}{\tilde \rho}\right)\partial_x^{\alpha-\beta}\Delta {\bf w}^{\neq}\cdot \partial_x^{\alpha}{\bf w}^{\neq}dx
\notag\\
&-\int_{\Omega}\nabla\left(\frac{1}{\tilde \rho}\right)\cdot\partial_x^{\alpha}\left(\frac{z}{\rho}\nabla {\bf w}\right)^{\neq}\cdot  \partial_x^{\alpha}{\bf w}^{\neq}dx-\frac{1}{\tilde \rho}\int_{\Omega}\partial_x^{\alpha}\left(\frac{z}{\rho}\nabla {\bf w}\right)^{\neq}\cdot \nabla\partial_x^{\alpha}{\bf w}^{\neq}dx
	\notag\\
	&-\int_{\Omega}\frac{1}{\tilde \rho}\partial_x^{\alpha}\left(\nabla\left(\frac{z}{\rho}\right)\cdot\nabla {\bf w}\right)^{\neq}\cdot \partial_x^{\alpha}{\bf w}^{\neq}dx+\sum_{|\beta|=1, \beta \leq \alpha}\int_{\Omega}\partial_x^{\beta}\left(\frac{1}{\tilde \rho}\right)\partial_x^{\alpha-\beta}\left(\frac{z}{\rho}\Delta {\bf w}\right)^{\neq}\cdot\partial_x^{\alpha}{\bf w}^{\neq}dx.
\end{align*}
From \eqref{E:6.2} we have
\begin{align}
\|J_7\|_{L^1(\Omega)}\lesssim& (\delta+\|z^{\neq}\|_{W^{2,\infty}(\Omega)}) \|\nabla {\bf w}^{\neq}, \nabla^2{\bf w}^{\neq}, \nabla^{|\alpha|+1} {\bf w}^{\neq}\|_{L^{2}(\Omega)}^2
		\notag\\
		&+\|\nabla{\bf w}^{\neq}\|_{L^{\infty}(\Omega)}\|\partial_x^{\alpha}\nabla z^{\neq}\|_{L^{2}(\Omega)}\|\partial_x^{\alpha}{\bf w}^{\neq}\|_{L^{2}(\Omega)}
		\notag\\
	\lesssim&(\delta+\nu) \| \nabla {\bf w}^{\neq}, \nabla^2{\bf w}^{\neq}, \nabla^{|\alpha|+1} {\bf w}^{\neq}\|_{L^{2}(\Omega)}^2,\nonumber
\end{align}
which yields that 
\begin{align}\label{E:6.57}
	-\int_{\Omega}\partial_x^{\alpha}\left(\frac{1}{\rho}\Delta {\bf w}\right)^{\neq}\cdot \partial_x^{\alpha}{\bf w}^{\neq}dx\gtrsim \|\nabla\partial_x^{\alpha}{\bf w}^{\neq}\|_{L^{2}(\Omega)}^2-(\delta+\nu) \| \nabla{\bf w}^{\neq},  \nabla^{|\alpha|}{\bf w}^{\neq}\|_{L^{2}(\Omega)}^2.
\end{align}
Substituting \eqref{E:6.56}-\eqref{E:6.57} into \eqref{E:6.50}, we have 
	\begin{align}\label{E:6.58}
	&\frac{d}{dt}\int_{\Omega}\left(\frac12(\nabla{\bf w}^{\neq})^2+\frac{T}{2\tilde \rho^2}\left(\nabla z^{\neq}\right)^2\right)dx+\|\nabla^2 {\bf w}^{\neq}\|_{L^{2}(\Omega)}^2
	\notag\\
	\lesssim&(\delta+\nu)\|\nabla z^{\neq}, \nabla {\bf w}^{\neq}\|_{L^2(\Omega)}^2+C_{\epsilon}\|\nabla H^{\neq}\|_{L^2(\Omega)}^2.
\end{align}

	Moreover, multiplying $\eqref{E:6.1}_2$ by $\cdot \nabla z^{\neq}$ gives that 
\begin{align*}
	\frac{T}{\tilde \rho}(\nabla z^{\neq})^2=&\left(\frac{Tz}{\rho}\nabla z\right)^{\neq}\cdot \nabla z^{\neq}-	({\bf w}^{\neq}\cdot \nabla z^{\neq})_t+\dive({\bf w}^{\neq}z_t^{\neq})+\dive {\bf w}^{\neq}\dive \left(\tilde \rho {\bf w}^{\neq}+{\bf J}_1\right)
	\notag\\
	&+(-{\bf h}_1^{\neq}-{\bf h}_2^{\neq})\cdot \nabla z^{\neq}-\nabla H^{\neq}\cdot \nabla z^{\neq}+\frac{\Delta {\bf w}^{\neq}}{\tilde \rho}\cdot \nabla z^{\neq}
	-\left(\frac{z}{\rho}\Delta {\bf w}\right)^{\neq}\cdot \nabla z^{\neq},
\end{align*}
then we use Lemma \ref{L:6.1} to get 
\begin{align}\label{E:6.60}
	\int_{\Omega}\frac{T}{2\tilde \rho}(\nabla z^{\neq})^2dx\lesssim \frac{d}{dt}\int_{\Omega}{\bf w}^{\neq}\cdot \nabla z^{\neq}dx+\|\nabla {\bf w}^{\neq}, \nabla^2 {\bf w}^{\neq}, \nabla H^{\neq}\|_{L^2(\Omega)}^2.
\end{align}
Taking the procedure as $C_{10}\eqref{E:6.58}+\eqref{E:6.60}$ with some positive constant $C_{10}$ large enough leads to \eqref{E:6.61}. Thus, the proof is completed.
\end{proof}

From Lemmas \ref{L:6.3}-\ref{L:6.5} we have
\begin{lemma}\label{L:6.6}
Under the assumptions of Theorem \ref{theorem1.1}, it holds that
\begin{align}\label{E:6.67a}
		&\frac{d}{dt}E_3(t)+\|\nabla H^{\neq}, \nabla H_t^{\neq}, \nabla {\bf w}^{\neq}\|_{H^1(\Omega)}^2+\| \nabla z^{\neq}\|_{L^2(\Omega)}^2
		\lesssim 0,
\end{align}
where $E_3(t)$ satisfies
	\begin{align}\label{E:6.67av}
\|\nabla H^{\neq}, H_t^{\neq}, z^{\neq}, {\bf w}^{\neq}\|_{H^1(\Omega)}^2\lesssim E_3(t)\lesssim  \|\nabla H^{\neq}, H_t^{\neq}, z^{\neq}, {\bf w}^{\neq}\|_{H^1(\Omega)}^2.
\end{align}
\end{lemma}
\begin{proof}
	Taking the procedure as $C_{12}\eqref{E:6.27}+C_{11}\eqref{E:6.49}+\eqref{E:6.61}$ and using the smallness of $\delta, \nu$, we have
		\begin{align}
		&\frac{d}{dt}E_3(t)+F_3(t)
		\lesssim 0,\nonumber
	\end{align}
where 
\begin{align*}
	E_3(t)=&C_{12}(E_1(t)+C_9E_2(t))+\int_{\Omega}C_{11}\left(\frac12({\bf w}^{\neq})^2+\frac{T}{2\tilde \rho}(z^{\neq})^2\right)dx
	\notag\\
	&+\int_{\Omega}\left(\frac12(\nabla{\bf w}^{\neq})^2+\frac{T}{2\tilde \rho^2}\left(\nabla z^{\neq}\right)^2-{\bf w}^{\neq}\cdot \nabla z^{\neq}\right)dx,
	\notag\\
	F_3(t)=&\frac{(C_{12}-C_{11}-1)}{2}\|\nabla H^{\neq}, \nabla H_t^{\neq}\|_{H^1(\Omega)}^2+\frac{(C_{11}-1)}{2}\|\nabla{\bf w}^{\neq}\|_{L^2(\Omega)}^2+\frac12\|\nabla z^{\neq}, \nabla^2 {\bf w}^{\neq}\|_{L^{2}(\Omega)}^2.
\end{align*}
It is easy to verify that there exists some positive constants $C_{11}$ and $C_{12}$ such that \eqref{E:6.67av}  
and 
	\begin{align*}
	F_3(t)\gtrsim  &\|\nabla H^{\neq}, \nabla H_t^{\neq}, \nabla {\bf w}^{\neq}\|_{H^1(\Omega)}^2+\| \nabla z^{\neq}\|_{L^2(\Omega)}^2
\end{align*}
hold. Thus, the proof is completed.
\end{proof}

As shown in Lemmas \ref{L:6.3} and \ref{L:6.5}, by taking the procedure as $\int_{\Omega}\partial_x^{\alpha}\eqref{E:6.1}_2\cdot \partial_x^{\alpha}{\bf w}^{\neq}dx$ with the multi-index $|\alpha|=2$, $\int_{\Omega}\partial_x^{\beta}\eqref{E:6.1}_2 \cdot \partial_x^{\beta}\nabla z^{\neq}dx$ as well as $\int_{\Omega}\partial_x^{\beta}\left(\tilde \rho\eqref{E:6.1}_2\right)\cdot (\partial_x^{\beta}\nabla H^{\neq}+\bar C_{\beta} \cdot \partial_x^{\beta}\nabla H_t^{\neq})dx$  with the multi-index $|\beta|=1$, we further prove the following lemma.
\begin{lemma}\label{L:6.7}
Under the assumptions of Theorem \ref{theorem1.1}, it holds that
	\begin{align}\label{iu5}
		&\frac{d}{dt}E_{4}(t)+\|\nabla^2 z^{\neq}, \nabla^3 {\bf w}^{\neq}\|_{L^{2}(\Omega)}^2+\|\nabla^2 H^{\neq}, \nabla^2 H_t^{\neq}\|_{H^1(\Omega)}^2
		\lesssim \| \nabla z^{\neq}, \nabla H^{\neq}\|_{L^2(\Omega)}^2+\|\nabla {\bf w}^{\neq}\|_{H^1(\Omega)}^2,
	\end{align}
	where $E_4(t)$ satisfies 
		\begin{align}\label{iu7}
		&\|\nabla^2 z^{\neq}, \nabla^2 {\bf w}^{\neq}\|_{L^2(\Omega)}^2+\|\nabla^2 H^{\neq}, \nabla H_t^{\neq}\|_{H^1(\Omega)}^2-\|\nabla {\bf w}^{\neq}\|_{L^2(\Omega)}^2-(\delta+\nu)\|\nabla z^{\neq}\|_{L^2(\Omega)}^2
		\notag\\
		\lesssim &E_{4}(t)
		\lesssim \|\nabla^2 z^{\neq}, \nabla^2 {\bf w}^{\neq}\|_{L^2(\Omega)}^2+\|\nabla^2 H^{\neq}, \nabla H_t^{\neq}\|_{H^1(\Omega)}^2+\|\nabla {\bf w}^{\neq}\|_{L^2(\Omega)}^2+(\delta+\nu)\|\nabla z^{\neq}\|_{L^2(\Omega)}^2.
	\end{align}
\end{lemma}

Combining Lemma \ref{L:6.6} with Lemma \ref{L:6.7} we get
\begin{lemma}
Under the assumptions of Theorem \ref{theorem1.1}, it holds that
\begin{align}\label{E:6.77i}
	\|H^{\neq}\|_{H^4(\Omega)}^2+\|H_t^{\neq}\|_{H^3(\Omega)}^2+\|z^{\neq}, {\bf w}^{\neq}\|_{H^2(\Omega)}^2\lesssim C_0e^{-ct}
\end{align}
for some positive constant $c$.
\end{lemma}
\begin{proof}
Taking the procedure as $C_{15}\eqref{E:6.67a}+\eqref{iu5}$ with some positive constant $C_{15}$ large enough yields that
\begin{align*}
\frac{d}{dt}E(t)+\|\nabla H^{\neq}, \nabla H_t^{\neq}, \nabla {\bf w}^{\neq}\|_{H^2(\Omega)}^2+\| \nabla z^{\neq}\|_{H^1(\Omega)}^2
		\lesssim 0,
\end{align*}
where $E(t):=C_{15}E_3(t)+E_4(t)$.
It follows from \eqref{E:6.67av} and \eqref{iu7} that 
\begin{align*}
\|H^{\neq}\|_{H^3(\Omega)}^2+\|H_t^{\neq}, z^{\neq}, {\bf w}^{\neq}\|_{H^2(\Omega)}^2\lesssim E(t)
\lesssim  \|H^{\neq}\|_{H^3(\Omega)}^2+\|H_t^{\neq}, z^{\neq}, {\bf w}^{\neq}\|_{H^2(\Omega)}^2.
\end{align*}
In addition, from \eqref{E:6.a} we have
\begin{align*}
\|\nabla H^{\neq}, \nabla H_t^{\neq}, \nabla {\bf w}^{\neq}\|_{H^2(\Omega)}^2+\| \nabla z^{\neq}\|_{H^1(\Omega)}^2\gtrsim \|\nabla H^{\neq}, H_t^{\neq}, z^{\neq}, {\bf w}^{\neq}\|_{H^2(\Omega)}^2\gtrsim E(t).
\end{align*}
Thus there exists a positive constant $c_1$ such that 
\begin{align*}
\frac{d}{dt}E(t)+c_1E(t)
	\leq 0,
\end{align*}
which yields that
\begin{align*}
E(t)\leq E(0)e^{-c_1t}\lesssim \nu_0 e^{-c_1t}.
\end{align*}
Moreover, from \eqref{E:6.1}-\eqref{E:6.2b} and \eqref{E:6.6a}-\eqref{E:6.6} we get
\[\|\Delta H^{\neq}\|_{H^2(\Omega)}^2+\|\Delta H_t^{\neq}\|_{H^1(\Omega)}^2\lesssim \|z^{\neq}, {\bf w}^{\neq}\|_{H^2(\Omega)}^2+\|\nabla H^{\neq}, \nabla H_t^{\neq}\|_{H^1(\Omega)}^2,\]
together with the fact that $\|\Delta H^{\neq}, \Delta H_t^{\neq}\|_{L^2(\Omega)}\gtrsim c_0 \|\nabla^2 H^{\neq}, \nabla^2 H_t^{\neq}\|_{L^2(\Omega)}$,
which implies \eqref{E:6.77i}. Thus, the proof is completed.
\end{proof}

\begin{proof}[\bf Proof of Theorem \ref{theorem1.1}]
It suffices to prove \eqref{E:1.17}. Applying the same argument in \eqref{E:6.4a}, we use \eqref{E:6.77i} to get
\begin{align*}
&\|\rho^{\neq}, {\bf u}^{\neq}\|_{W^{1, \infty}(\Omega)}+\|E_t^{\neq}\|_{W^{2, \infty}(\Omega)}+\|E^{\neq}\|_{W^{3, \infty}(\Omega)}
\notag\\
=&\|z^{\neq}, {\bf w}^{\neq}\|_{W^{1, \infty}(\Omega)}(t)+\|H_t^{\neq}\|_{W^{2, \infty}(\Omega)}(t)+\|H^{\neq}\|_{W^{3, \infty}(\Omega)} 
\notag\\
\lesssim &\|\nabla(z^{\neq}, {\bf w}^{\neq})\|_{H^1(\Omega)}^{\frac{1}{2}}+\|\nabla H_t^{\neq}\|_{H^2(\Omega)}^{\frac{1}{2}}+\|\nabla H^{\neq}\|_{H^3(\Omega)}^{\frac{1}{2}}\lesssim e^{-\frac{c_1}{4}t}.
\end{align*}
Since 
\begin{align*}
{\bf m}^{\neq}=(\rho{\bf u})^{\neq}=(\rho^{\neq}{\bf u}^{\neq})^{\neq}+\bar {\bf u}\rho^{\neq}+\bar \rho{\bf u}^{\neq},
\end{align*}
we use the boundness of $\|\rho, {\bf u}\|_{W^{1,\infty}(\Omega)}$ to deduce
\eqref{E:1.17}. Thus, the proof is completed.
\end{proof}

 \noindent
{\bf Acknowledgements\ }
X. Wu's research is supported in part by the National Natural Science Foundation of China (No. 12201649) and Hunan Provincial Natural Science Foundation of China
(2023JJ40699).

     \end{document}